\pdfoutput=1
\documentclass[11pt]{article}

\usepackage[utf8]{inputenc}
\usepackage[T1]{fontenc}
\usepackage[margin=1in]{geometry}
\usepackage{microtype}
\usepackage{amsfonts,amsmath,amssymb,amsthm}
\usepackage{graphicx}
\usepackage{wrapfig}
\usepackage{xcolor}
\usepackage[round,authoryear]{natbib}
\usepackage[colorlinks=true,linkcolor=red!60!black,citecolor=blue!60!black,urlcolor=orange!70!black]{hyperref}
\usepackage[capitalize]{cleveref}
\usepackage{authblk}

\newtheorem{lemma}{Lemma}
\newtheorem{proposition}{Proposition}
\newtheorem*{proposition*}{Proposition}
\newtheorem{theorem}{Theorem}
\newtheorem*{theorem*}{Theorem}

\newcommand{\A}{\mathcal A}
\newcommand{\Bernoulli}{\operatorname{Ber}}
\newcommand{\crl}[1]{\left\{#1\right\}}
\newcommand{\dd}{\mathrm d}
\newcommand{\E}{\mathbb E}
\newcommand{\Ee}{\mathcal E}
\newcommand{\KL}[2]{\operatorname{KL}(#1 || #2)}
\newcommand{\N}{\mathcal N}
\newcommand{\Nn}{\mathbb N}
\newcommand{\one}{\mathbf{1}}
\newcommand{\PR}{\mathcal P}
\newcommand{\R}{\mathbb R}
\newcommand{\Sc}{\mathcal S}
\newcommand{\X}{\mathcal X}

\title{\vspace{-2em}Time-sensitive anytime-valid testing}

\author[$a$,$\ast$]{Eugenio Clerico}
\author[$b$,$\ast$]{Tobias Wegel}
\author[$a$]{Iskander Azangulov}
\author[$a$]{Patrick Rebeschini}

\affil[$a$]{\small Department of Statistics, University of Oxford}
\affil[$b$]{\small Department of Computer Science, ETH Z\"urich}

\date{}

\makeatletter
\renewcommand{\@maketitle}{%
  \newpage
  \null
  \begin{center}%
    {\huge \@title \par}%
    \vskip 2.5em%
    {\large
      \lineskip .5em%
      \begin{tabular}[t]{c}%
        \AB@authlist
      \end{tabular}\par}%
    \vskip .2em%
    {\small
      \begin{tabular}[t]{c}%
        \AB@affillist
      \end{tabular}\par}%
  \end{center}%
  \par
  \vskip 1.5em%
}
\makeatother

\begin{document}

\maketitle

\begingroup
\renewcommand\thefootnote{$\ast$}
\footnotetext{Contributed equally}
\endgroup

\begin{abstract}
\noindent Anytime-valid tests allow evidence to be checked during data collection: one can either continue testing or stop and reject the null while still controlling type-I error. Yet, in many applications rejection is useful only if it comes soon enough. We introduce a time-sensitive testing-by-betting framework that favours early rejection by assigning rewards to rejection times and maximising their expected value under a given alternative. This encompasses hard deadlines and softer time preferences. The resulting optimal control problem admits a Bellman representation in terms only of time and evidence against the null, rather than the full history. For hard deadlines,  the simple-vs-simple case reduces to a finite-horizon Neyman--Pearson problem and identify the corresponding optimal e-process. Furthermore, we show that exponentially decaying rewards admit a stationary approximation, yielding the exponential-decay-optimal (EDO) criterion: a finite-time-scale counterpart to the classical growth-rate-optimal (GRO) viewpoint in anytime-valid statistics, with the GRO criterion recovered in the large-time-scale limit.
\end{abstract}

\section{Introduction}
Classical hypothesis testing in statistics is built around fixed-sample designs: one specifies a sample size in advance, then collects the data and computes a p-value. However, this framework is often too rigid for modern settings, where data arrive sequentially and may be monitored continuously. Practitioners increasingly need the freedom to adaptively stop or continue testing based on real-time evidence. As a consequence, recent literature has focused on anytime-valid inference, developing tools, such as e-values and e-processes, that provide statistical guarantees which remain valid under optional stopping and continuation \citep{grunwald2024safe,ramdas2025hypothesis}. A concrete and intuitive instantiation of anytime-valid inference is the \emph{testing-by-betting} game \citep{ramdas2023game}, where a gambler starts with unit wealth and, at each round, places a bet against the null hypothesis. The game is structured to be fair (or  disadvantageous to the player) when the null is true, meaning that the gambler is highly unlikely to become rich. Consequently, accumulating significant wealth constitutes evidence against the null. The appeal of this perspective is that it turns sequential testing into an algorithmic problem: based on currently available observations, one needs to decide where and how aggressively to bet. In this sense, the statistical question naturally admits a sequential-control interpretation, creating a bridge towards dynamic programming and reinforcement learning.

To guide betting decisions in the testing-by-betting game, the statistical literature typically fixes an alternative hypothesis and chooses bets that maximise asymptotic wealth when that alternative holds. This is formalised by the growth-rate-optimal (GRO) criterion \citep{grunwald2024safe}, which in simple cases can recover classical ideas such as the Kelly criterion \citep{kelly1956new} and Wald's sequential likelihood-ratio test (SPRT) \citep{wald1945sequential}. This approach is natural and theoretically  justified, especially in idealised settings where no time or cost constraints affect the experiment. Yet, from a control perspective, long-run growth is only one possible objective, and in many practical applications  arguably not the most relevant one. Often, a guarantee of eventual rejection is insufficient, with practical considerations requiring the test to reject as early as possible. It is thus natural to think of the \emph{value} of a rejection as depreciating over time. Recent work has  tackled this or related issues by  focusing on finite-horizon settings, showing that when a fixed deadline is imposed, asymptotically optimal strategies can be strictly suboptimal \citep{voravcek2025star,taga2026learning}.

The starting point of our work is that the finite-horizon formulation addresses only one case within a fundamentally broader control problem.

Most sequential experiments do not naturally come with a single hard deadline, yet there are natural preferences over rejection times: earlier rejection is better, but later rejection may still retain some value. This observation suggests shifting the focus from a fixed horizon to a reward assigned directly to the rejection time. In this framework, the hard-deadline scenario is recovered in the special case where the reward is an indicator that vanishes after a given step. Alternative reward choices allow softer and more expressive formulations of time preference. Among valid betting strategies, we therefore propose to consider those that maximise the expected reward associated with the rejection time under the given alternative. \cref{fig:heatmaps-Gaussian} shows how, in a simple Gaussian testing problem, changing the reward yields different optimal betting strategies, hence inducing different distribution of rejection times. We dub our approach \emph{time-sensitive anytime-valid testing}, and view it as the natural control formulation of sequential testing when rejection times matter. In summary, the main conceptual contribution of this paper is to formulate a reward-based sequential decision problem that naturally encodes  preferences over the duration of the experiment, all while preserving anytime validity. For the sake of clarity, we develop our formulation for i.i.d.~testing-by-betting, although we expect it to  extend to more general statistical frameworks.

\begin{figure}[b]
    \vspace{-1.5em}
    \centering
    \includegraphics[width=\linewidth]{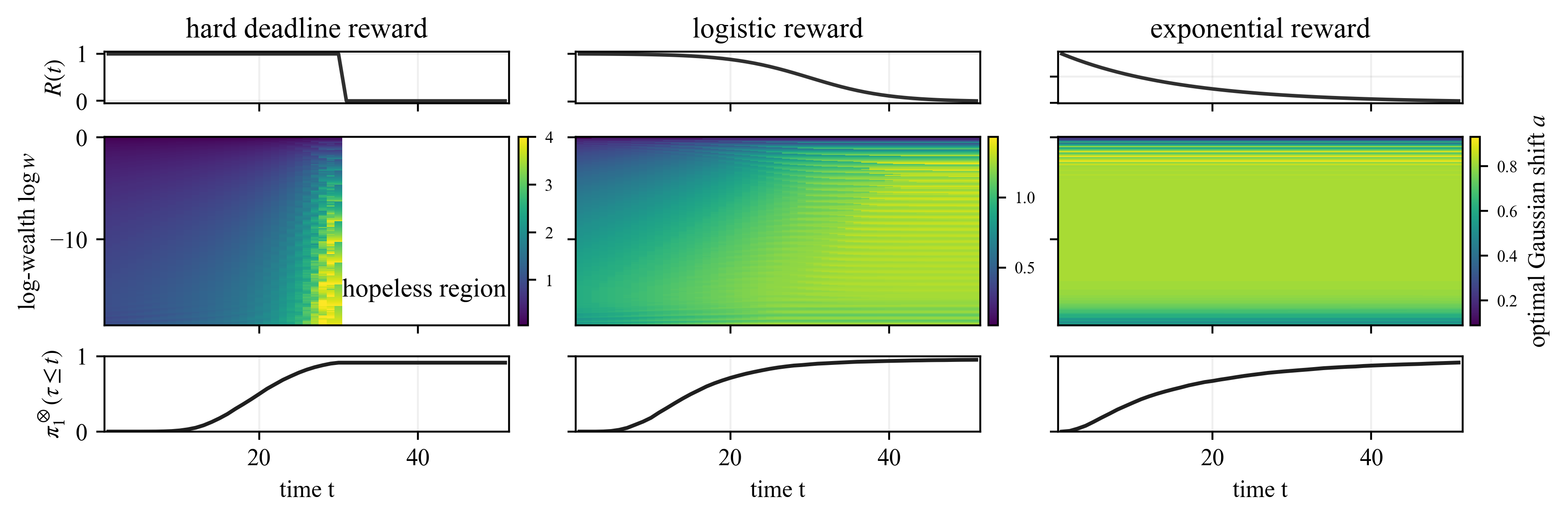}
    \caption{\small Heatmap of the optimal action $a$, representing the bet $E_{[a]}(x)=\exp\{ax-a^2/2\}$, as a function of time $t$ and log-wealth $\log w$ when testing $\pi_0=\N(0,1)$ against $\pi_1=\N(0.6,1)$ with $\alpha=.05$, under different  rewards. The corresponding stopping-time cdfs are shown below. Further details in \Cref{sec:app-fig-heatmaps-gaussian}.}
    \label{fig:heatmaps-Gaussian}

\end{figure}

The reward-based viewpoint that we introduce brings the control perspective to the forefront. Although testing by betting is already sequential in nature, assigning rewards to rejection times makes the choice of a betting strategy an explicit control problem over time and wealth. This viewpoint is in line with recent reinforcement-learning approaches to horizon-aware anytime-valid testing \citep{taga2026learning}, but differs in that we optimise against a prescribed alternative and allow general rewards over rejection times rather than only a fixed deadline. The same target-reaching structure also relates our work to a longer tradition of optimal gambling, including red-and-black and Dubins--Savage gambling, where stakes are chosen as a function of current fortune in order to reach a target wealth, with finite-time and discounted variants studied in later work \citep{dubins1965gamble,yao2007bold}. Our setting has a similar target-reaching flavour, yet also fundamental differences, as  both the admissible bets and the utility are dictated by the statistical testing problem, with validity under the null constraining the betting rules and the prescribed alternative determining the reward objective.

Our main contributions are threefold. First, we show that, among all predictable betting strategies that can depend on the full observation history, it is enough to consider policies depending only on the current wealth and round, which yields  a simpler  Bellman formulation. Second, in the simple-versus-simple setting, we show that restricting attention to testing-by-betting incurs no loss compared with optimising over more general anytime-valid procedures (testing via e-processes). Third, we analyse hard-deadline and exponentially decaying rewards. The hard-deadline reward recovers a finite-horizon Neyman--Pearson structure in the simple-versus-simple setting. For exponentially decaying rewards, we derive the exponential-decay optimal (EDO) criterion, which selects a stationary constant bet linked to power utilities and Rényi divergences, and recovers the GRO log-utility criterion in the large-time-scale limit. We illustrate our framework with Gaussian and Bernoulli examples.

\paragraph{Paper structure.}
The rest of the paper is organised as follows. \Cref{sec:control} introduces the sequential-control formulation of time-sensitive anytime-valid testing. We then prove the Markovian reduction and derive the associated Bellman representation in \Cref{sec:markov}. \Cref{sec:simple} specialises to the simple-vs-simple setting, where testing-by-betting with a canonical action space is shown to be equivalent to optimisation over general e-processes. \Cref{sec:hard} analyses hard-deadline rewards and relates the optimal procedure to a finite-horizon Neyman--Pearson problem, while \Cref{sec:exponential-reward} studies exponentially decaying rewards and introduces the EDO criterion. \Cref{sec:examples} works out Gaussian and Bernoulli examples. A brief collection of notation and conventions used throughout the paper is given in \Cref{app:notation}. All proofs are detailed in \Cref{app:proofs}. Additional calculations for the examples are provided in \Cref{sec:app-examples}, while \Cref{sec:appendix-experiments} describes the numerical experiments used to generate the figures. Finally, \Cref{sec:appendix-hard} gives further discussion of the hard-deadline reward and compares different horizon-aware methods.

\section{Testing-by-betting and time-sensitive formulation}
We let $\X$ be a standard Borel space and write $\PR_\X$ for the set of Borel probability measures on $\X$. We fix a null hypothesis $\PR_0\subseteq\PR_\X$ and a \emph{simple} alternative $\pi_1\in\PR_\X\setminus\PR_0$. We consider the problem of testing $\PR_0$ against $\pi_1$ in the i.i.d.~setting. More precisely, we sequentially observe a stream of observations $x_1, x_2,\dots$ and wish to test whether these are i.i.d.~draws from some unknown $\pi\in\PR_0$, against the alternative that they are i.i.d.~draws from $\pi_1$. It is convenient to formulate the problem directly on the path space $\X^\infty$ (endowed with the usual product $\sigma$-algebra). For $\pi\in\PR_\X$ we write $\pi^\otimes$ for the corresponding infinite product measures, and set $\PR_0^\otimes = \{\pi^\otimes\,:\,\pi\in\PR_0\}$. Our goal is therefore to test the sequential null  $\PR_0^\otimes$ against the simple alternative $\pi_1^\otimes$. Throughout the paper, we fix a significance level $\alpha\in(0,1)$: we are  interested in procedures whose type-I error  is at most $\alpha$.

We call an \emph{e-variable} for $\PR_0$ any Borel function $E:\X\to\R_+$ such that $\E_\pi[E]\leq 1$ for every $\pi\in\PR_0$. E-variables serve as building blocks of modern anytime-valid testing \citep{ramdas2023game, ramdas2025hypothesis} and as the core ingredient of the following \emph{testing-by-betting} framework, where they represent bets against the null. Fix a collection $\Ee$ of e-variables. A player starts with wealth $W_0=1$ and, at each round $t\geq1$, chooses predictably (based on past observations) a bet $E_t\in\Ee$. Then $x_t$ is revealed and the wealth  updated as $W_t=W_{t-1}E_t(x_t)$. Under the null, $(W_t)_{t\geq0}$ is a non-negative supermartingale. So, regardless of the betting strategy, the player is unlikely to accumulate substantial wealth, and rejecting the null once $W_t\geq\alpha^{-1}$ yields a valid level-$\alpha$ test. This is formalised in the next proposition (Proposition 1 in \citealp{clerico2025optimality}),  a direct consequence of Ville's inequality.
\begin{proposition}\label{prop:testbybet}
    For any predictable betting strategy in the testing-by-betting game, the procedure that rejects when the wealth first reaches $\alpha^{-1}$ is a valid sequential test at level $\alpha$. That is, when i.i.d.~draws from some $\pi\in\PR_0$ are observed,  the procedure rejects with probability at most $\alpha$.
\end{proposition}

A central question is how to choose the bets $E_t$. When an alternative $\pi_1$ is specified (i.e., under $\pi_1$  one wants the test to reject efficiently) the standard recommendation is the growth-rate-optimal (GRO) criterion \citep{grunwald2024safe}: choose an e-variable maximising the \emph{e-power} $\E_{\pi_1}[\log E]$.  This objective is the testing-by-betting analogue of the Kelly criterion \citep{kelly1956new}, and provides a standard notion of asymptotic power under optional continuation \citep{ramdas2025hypothesis}. For a \emph{simple} null (i.e., $\PR_0=\{\pi_0\}$ is a single measure), the GRO bet coincides with the likelihood ratio $\dd\pi_1/\dd\pi_0$,  recovering Wald's sequential probability ratio test (SPRT) \citep{wald1945sequential}. Its statistical justification is fundamentally asymptotic: under a constant strategy $E_t=E$, the law of large numbers gives $\log W_t\approx t\,\E_{\pi_1}[\log E]$ for large $t$ under $\pi_1^\otimes$. Hence, if $\E_{\pi_1}[\log E]>0$, wealth grows exponentially and the null is rejected in finite time almost surely. Moreover, \cite{agrawal2025stopping} showed that, as $\alpha\to0$, the rejection time has first-order asymptotic scale $\log(1/\alpha)/\E_{\pi_1}[\log E]$. Thus, maximising e-power minimises the first-order asymptotic scale of the rejection time.

Since the GRO criterion controls the rejection time only through its first-order asymptotic scale, it need not reflect preferences over when rejection occurs. In many applications, a test that rejects eventually may still be unsatisfactory if it does so too late. This suggests formulating the problem directly in terms of the \emph{rejection time} $\tau=\inf\{t\geq0:W_t\geq\alpha^{-1}\}$, with $\tau=\infty$ if the threshold is never reached. To express preferences over rejection times, we fix a non-increasing \emph{reward function} $R:\mathbb{N}\cup\{\infty\}\to\R_+$, with $R(\infty)=0$. We evaluate the performance of a betting strategy in the testing-by-betting game by $\E_{\pi_1^\otimes}[R(\tau)]$: we aim at a betting strategy that maximises this expected reward, thus replacing the asymptotic GRO criterion with an objective directly accounting for rejection time. A few natural examples of reward functions are the following. $R(t)=\one_{\{t\leq T\}}$ corresponds to maximising the probability of rejection within the first $T$ rounds, as in the hard-deadline setting of \cite{taga2026learning}. A softer version is the logistic $R(t)=1/(1+e^{(t-T)/\beta})$, with $\beta>0$, which smooths the transition around  $T$. Another  option is $R(t)=e^{-t/T}$,  assigning an exponentially decaying value to rejection, with time scale $T>0$. These examples are illustrated in \cref{fig:heatmaps-Gaussian}.

\section{Sequential control formalisation}\label{sec:control}

We now formalise the testing-by-betting procedure as a sequential control problem.  We consider an \emph{action space} $\A$ standard Borel and a Borel map $\phi:\A\times\X\to\R_+$ satisfying the validity condition
$$\int_\X\phi(a,x)\,\dd\pi(x)\leq1\,,$$
for every $a\in\A$ and $\pi\in\PR_0$. Each  $a\in\A$ indexes an e-variable $E_{[a]}=\phi(a,\,\cdot\,)$: in a testing-by-betting game with $\Ee = \{E_{[a]}\,:\,a\in\A\}$, ``\emph{choosing  action $a$}'' means ``\emph{playing  bet $E_{[a]}$}''.

In simple-null examples $\PR_0=\{\pi_0\}$, a natural construction is to let $a$ index a measure $\nu_a\ll\pi_0$ and set $\phi(a,x)=\dd\nu_a/\dd\pi_0(x)$ (cf.~\Cref{sec:simple}). For Bernoulli testing with $\pi_0=\Bernoulli(p_0)$, for $\A=[0,1]$ and $\nu_a=\Bernoulli(a)$, $\phi(a,1)=a/p_0$ and $\phi(a,0)=(1-a)/(1-p_0)$. For Gaussian testing with $\pi_0=\N(0,\sigma^2)$, taking $\A=\R$ and $\nu_a=\N(a,\sigma^2)$ gives $\phi(a,x)=\exp(ax/\sigma^2-a^2/(2\sigma^2))$. Beyond simple nulls, natural parametrisations are often available. For instance, for $\PR_0$ the class of measures supported on $[-1,1]$ with mean $m$, one obtains the coin-betting class \citep{orabona2024tight, clerico2025optimality} by taking $\A=[(m-1)^{-1},m^{-1}]$ and $\phi(a,x)=1+a(x-m)$. More generally, when the null is specified through constraints, the corresponding dual representation of e-variables yields a natural parametrised action space \citep{clerico2024properly,larsson2025constraints}.

We call a \emph{policy} any sequence $A^\infty=(A_t)_{t\geq1}$, where each $A_t:\X^{t-1}\to\A$ is a Borel map. At round $t$, after observing $x^{t-1}$, the policy plays the e-variable $E_t=E_{[A_t(x^{t-1})]}=\phi(A_t(x^{t-1}),\,\cdot\,)$. Define recursively
$$\widehat W_t^{A^\infty}(x^t)=\widehat W_{t-1}^{A^\infty}(x^{t-1})\,\phi\big(A_t(x^{t-1}),x_t\big)=\prod_{k=1}^t\phi\big(A_k(x^{k-1}),x_k\big) = \prod_{k=1}^t E_k(x_k)\,,$$ with $\widehat W_0^{A^\infty}(x^0)=1$. This is the \emph{wealth} at time $t$ under the policy $A^\infty$. We define $\tau^{A^\infty}:\X^\infty\to\Nn\cup\{\infty\}$ as the first time the wealth under $A^\infty$ reaches the  threshold $\alpha^{-1}$, namely
$$\tau^{A^\infty}(x^\infty) = \inf\{t\geq 1\,:\,\widehat W_t^{A^\infty}(x^t)\geq\alpha^{-1}\}\,,$$
where $\inf\emptyset=\infty$. $\tau^{A^\infty}$ is  a stopping time with respect to the natural filtration induced by the sequence of observations. With these definitions, we have reformulated the testing-by-betting game as a sequential control problem with action space $\A$, where the available bets are precisely $\{E_{[a]}:a\in\A\}$. Hence, by \Cref{prop:testbybet}, the procedure that rejects at time $\tau^{A^\infty}$ is valid at level $\alpha$.

Recall a \emph{reward function} is any non-increasing map $R:\Nn\cup\{\infty\}\to\R_+$ with $R(\infty)=0$. Evaluating a policy through the expected reward of its rejection time under $\pi_1^\otimes$ leads to the objective
\begin{equation}\label{eq:objective}\sup_{A^\infty}\E_{\pi_1^\otimes}[R(\tau^{A^\infty})]\,.\end{equation}

\section{Markovian reduction and Bellman formulation}\label{sec:markov}
The control formulation that we introduced allows very general policies: the action chosen at time $t+1$ may depend on the entire past $x^t$. This raises a natural question: is such full history dependence  needed? Our first main structural result shows that it is not. In fact, for the reward objective that we consider, there is no loss in restricting  to policies whose decision at time $t+1$ depends on the past only through the wealth. We call $a^\infty=(a_t)_{t\geq0}$ a \emph{Markov policy} if each $a_t:[0,\alpha^{-1})\to\A$ is Borel. We identify such a sequence with any ordinary policy that chooses the action $a_t(w)$ at round $t$ if the current wealth is $w<\alpha^{-1}$. Notably, once the threshold is reached the choice of action is irrelevant. With slight abuse of notations, we write $\tau^{a^\infty}$ for the rejection time induced by the Markov policy.

\begin{theorem}
\label{thm:markovian_reduction}
For any non-increasing reward function $R:\Nn\cup\{\infty\}\to\R_+$ we have  $$\sup_{A^\infty}\E_{\pi_1^\otimes}[R(\tau^{A^\infty})]=\sup_{a^\infty\text{ Markov}}\E_{\pi_1^\otimes}[R(\tau^{a^\infty})]\,.$$\end{theorem}

Although intuitive, the Markovian reduction is not straightforward  at this level of generality. Yet, this generality matters:  the infinite-dimensional class of e-variables  might not admit a natural finite-dimensional parametrisation without loss of power (e.g., \Cref{sec:simple}  shows that for simple-vs-simple tests the canonical action space is the space of all probability measures dominated by the null). \Cref{thm:markovian_reduction}'s proof embeds the problem in a positive Borel dynamic programme whose state records time, wealth, and rejection,  then concludes via a stationary-policy result  \citep{schal1987stationary}.

Since the restriction of a Markov policy to future time steps remains Markov, it is natural to consider continuation problems starting from an arbitrary wealth $w\in[0,\alpha^{-1})$. For a Markov policy $a^\infty$, let
$$\widetilde W_t^{w,a^\infty}(x^t)=\widetilde W_{t-1}^{w,a^\infty}(x^{t-1})\,\phi\bigl(a_{t-1}(\widetilde W_{t-1}^{w,a^\infty}(x^{t-1})),x_t\bigr)\,,$$
where formally $\widetilde W_0^{w,a^\infty}(x^0)=w$. We also define $\tau_w^{a^\infty}(x^\infty)=\inf\{t\geq 1\,:\,\widetilde W_t^{w,a^\infty}(x^t)\geq \alpha^{-1}\}$, the number of additional rounds needed for rejection when the current wealth is $w\in[0,\alpha^{-1})$ and future actions follow $a^\infty$. This continuation viewpoint naturally leads to the value function
$$V_t(w)=\sup_{a^\infty\text{ Markov}}\E_{\pi_1^\otimes}\bigl[R(t+\tau_w^{a^\infty})\bigr]\,.$$
In words, $V_t(w)$ is the maximal expected reward that can still be achieved when $t$ rounds have already been played, the current wealth is $w$, and the threshold has not yet been reached. Now \eqref{eq:objective} reads
$$\sup_{A^\infty}\E_{\pi_1^\otimes}[R(\tau^{A^\infty})]= \sup_{a^\infty\text{ Markov}}\E_{\pi_1^\otimes}[R(\tau^{a^\infty})] = V_0(1)\,.$$

For $t\geq0$ and $w\in[0,\alpha^{-1})$, we obtain formally  the following Bellman recursion:
\begin{equation}
\label{eq:bellman}
V_t(w)=\sup_{a\in\A}\int_\X\Bigl(R(t+1)\mathbf{1}_{\{w\phi(a,x)\geq \alpha^{-1}\}}+V_{t+1}(w\phi(a,x))\mathbf{1}_{\{w\phi(a,x)<\alpha^{-1}\}}\Bigr)\dd\pi_1(x)\,.
\end{equation}
We remark that, at this level of generality, the Bellman recursion should be interpreted with some care. Since $V_t$ is defined as a supremum, there is no reason a priori for $V_t$ itself to be Borel. Nevertheless, $V_t$ is at least universally measurable, which is enough to ensure that the integral in \eqref{eq:bellman} is well defined. However, in general, there is no guarantee of either a closed-form expression for $V_t$ or an optimal Borel policy attaining the supremum in \eqref{eq:bellman}. We therefore view the Bellman formulation primarily as a guiding principle. On the one hand, it provides a natural starting point for numerical approximation schemes, such as value-iteration type methods. On the other hand, it helps identify simpler but still informative approximations of the control problem itself, as we will see in \Cref{sec:exponential-reward}.

\section{Simple-vs-simple case: equivalence with e-processes}\label{sec:simple}
Before turning to concrete reward functions, it is worth relating  testing-by-betting  to the broader class of sequential tests based on e-processes of which testing-by-betting is a special case. Given a null hypothesis on $\X^\infty$, an \emph{e-process} is a sequence $U^\infty=(U_t)_{t\geq0}$ of non-negative Borel maps $U_t:\X^t\to\R_+$ such that $\E_{\Pi_0}[U_\tau]\leq1$ for every stopping time $\tau$ and every $\Pi_0$ in the null. Such a process induces a valid sequential test by rejecting when $U_t$ first reaches $\alpha^{-1}$ \citep{ramdas2025hypothesis}. Testing by betting constructs evidence processes by multiplying one-round bets, whereas general e-processes need not admit such a product representation \citep{ramdas2022testing}. It is  natural to ask whether, for our reward-based objective, restricting  to testing-by-betting can lose power.

We now specialise to the case of a simple null $\pi_0^\otimes$ and simple alternative $\pi_1^\otimes$, where $\pi_0,\pi_1\in\PR_\X$ and $\pi_1\ll\pi_0$. In this setting, there is a canonical choice of  action space $\A$ that entails no loss of power. First, note that every  e-variable $E$ for $\pi_0$ satisfies $\E_{\pi_0}[E]\leq 1$, and since the reward is non-increasing, any non-zero such $E$ may be normalised to satisfy $\E_{\pi_0}[E]=1$, which only increases the resulting wealth pathwise. Hence, without loss of generality, one may restrict attention to those e-variables satisfying $\E_{\pi_0}[E]=1$. These are precisely the Radon--Nikodym densities of probability measures dominated by $\pi_0$. Accordingly, in the simple-null case we may take as action space $\A$ the set
$$\PR^{\pi_0}=\{\nu\in\PR_\X\,:\,\nu\ll\pi_0\}\,.$$
\Cref{lemma:canonical_action_space} in \Cref{app:proofs} shows that $\PR^{\pi_0}$ is a standard Borel space and that versions of Radon--Nikodym densities can be chosen so that $(\nu,x)\mapsto \phi^{\pi_0}(\nu,x)=\dd\nu/\dd\pi_0(x)$ is Borel on $\PR^{\pi_0}\times\X$.\footnote{Since $\pi_1\ll\pi_0$, the particular choice of version does not affect either the null constraint or the objective under $\pi_1^\otimes$.} Thus, for simple nulls we may take $(\A,\phi)=(\PR^{\pi_0},\phi^{\pi_0})$. This is the full canonical version of the likelihood-ratio parametrisation used in the examples at the beginning of \Cref{sec:control}. In the Bernoulli case with $p_0\in(0,1)$, $\A=[0,1]$ coincides with $\PR^{\pi_0}$. By contrast, when $\pi_0=\N(0,\sigma^2)$, setting $\A=\R$ as in \Cref{sec:control} only parametrises the Gaussian-shift subfamily $\{\N(a,\sigma^2):a\in\R\}\subsetneq\PR^{\pi_0}$.

Let $\tau^{U^\infty}=\inf\{t\geq 0\,:\,U_t\geq \alpha^{-1}\}$ be the rejection time of $U^\infty$. The following equivalence holds.
\begin{theorem}
\label{thm:eprocess_equivalence}
Fix $\pi_0,\pi_1\in\PR_\X$, with $\pi_1\ll\pi_0$, and let $(\A,\phi) = (\PR^{\pi_0},\phi^{\pi_0})$. Then, we have
$$\sup_{U^\infty\text{ e-process}}\E_{\pi_1^\otimes}[R(\tau^{U^\infty})]=\sup_{A^\infty}\E_{\pi_1^\otimes}[R(\tau^{A^\infty})] = \sup_{a^\infty\text{ Markov}}\E_{\pi_1^\otimes}[R(\tau^{a^\infty})]\,.$$
\end{theorem}

The proof leverages that every e-process for a simple null is dominated by a  martingale \citep{ramdas2020admissible}. Thus, every e-process procedure can be realised as a testing-by-betting strategy with history-dependent one-round bets, and the Markovian reduction of \Cref{thm:markovian_reduction} then yields the claim.
\Cref{thm:eprocess_equivalence} shows that, in the simple-vs-simple i.i.d.\ setting, the control formulation loses no power from using one-round bets, provided $(\A,\phi)$ is chosen canonically. It also recovers Wald's classical SPRT under GRO (whose bet is $d\pi_1/d\pi_0$ each round), while our reward-based approach can yield more general policies that may depend on time and wealth.

\section{Hard-deadline reward}\label{sec:hard}

A natural choice for the reward is the hard-deadline $R(t)=\one_{\crl{t\leq T}}$. This yields the objective \eqref{eq:objective} $$\sup_{A^\infty}\E_{\pi_1^\otimes}[R(\tau^{A^\infty})] = \sup_{A^\infty}\pi_1^\otimes(\tau^{A^\infty}\leq T) =  \sup_{a^\infty\text{ Markov}} \pi_1^\otimes(\tau^{a^\infty}\leq T)\,,$$ namely maximising the probability of rejecting the null under the alternative before a fixed horizon. Related finite-horizon ideas appear in \cite{voravcek2025star} in the context of coin-betting for mean estimation (see, e.g., \citealp{orabona2024tight, waudbysmith2024estimating}). There, the deadline affects the betting strategy by inducing more aggressive bets as the horizon approaches, a behaviour one also expects in our formulation. The key difference is that their approach lacks a prescribed alternative, but implicitly generates  adaptive surrogate alternatives from observed deviations from the null. The horizon-aware formulation of \cite{taga2026learning} is closer to ours in its control objective: bounded mean testing under a fixed deadline is formulated as a finite-horizon control problem over time and wealth, aiming to maximise rejection probability before the deadline while preserving anytime-validity. However, unlike our formulation, they do not fix a single simple alternative. Instead, their reinforcement learning implementation uses several synthetic alternatives  as oracle training instances. As solving the dynamic programme exactly is generally hard, they mainly use its structure to describe regimes of betting aggressiveness and  motivate heuristic betting policies.

Even in our framework, with a fixed simple alternative $\pi_1$, the hard-deadline control problem remains intractable in general. Yet, in the simple-vs-simple case it  admits a canonical optimal Doob e-process (\cref{fig:bernoulli_doob}), whose construction is closely related to the induction of \cite{koning2026anytime}.
\begin{proposition}\label{prop:hard-deadline-np}
Consider testing $\pi_0^\otimes$ against $\pi_1^\otimes$, with $\pi_1\ll\pi_0$. Then there exists a Borel set $\Gamma^\star\subseteq\X^T$, viewed as an event depending only on the first $T$ observations, such that $\pi_0^{\otimes}(\Gamma^\star)\leq\alpha$ and
$$\sup_{U^\infty\text{ e-process}}\pi_1^\otimes\bigl(\tau^{U^\infty}\leq T\bigr)=\pi_1^{\otimes}(\Gamma^\star)=\max\bigl\{\pi_1^{\otimes}(\Gamma)\,:\,\Gamma\subseteq\X^T\text{ Borel}\,,\,\pi_0^{\otimes}(\Gamma)\leq\alpha\bigr\}\,.$$
Fixing Borel versions of the conditional probabilities, the supremum  is achieved by the $\pi_0^\otimes$-martingale
$U_t^\star(x^t)=\pi_0^\otimes(\Gamma^\star\mid x^t)/\pi_0^\otimes(\Gamma^\star)$ if $\pi_0^\otimes(\Gamma^\star)>0$, and by the constant e-process $U_t\equiv 1$ otherwise.
\end{proposition}

Here $\Gamma^\star$ is a Neyman--Pearson rejection region for the fixed-sample problem at horizon $T$ \citep{neyman1933problem,lehmann2005testing}. Notably, the hard-deadline reward does not favour rejection earlier than the deadline. This can make the induced Doob test effectively non-sequential, as in Gaussian mean testing, where it coincides with the fixed-horizon Neyman--Pearson test (cf.~\Cref{sec:examples}). In Bernoulli testing, earlier rejection is possible, but the procedure is still conservative until close to the deadline (see \Cref{sec:appendix-hard} for further discussion). This motivates considering other rewards, such as the exponential-decay reward of the next section.

\begin{figure}[!b]
    \centering
    \includegraphics[width=0.75\textwidth]{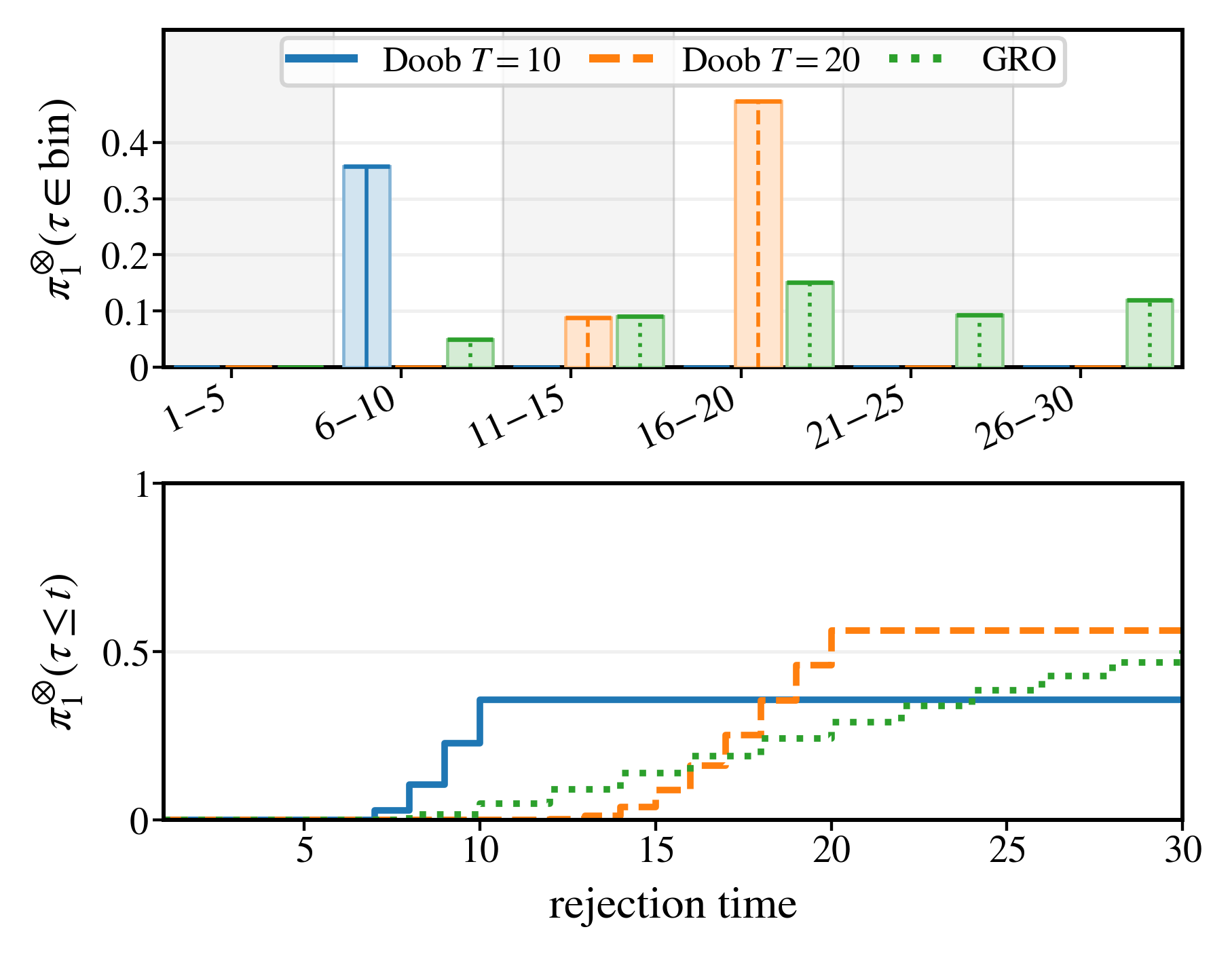}
    \caption{Rejection-time distributions for $\pi_0=\Bernoulli(0.4)$ vs $\pi_1=\Bernoulli(0.6)$ under the Doob e-process, compared with GRO. Details in \Cref{sec:appendix-fig-bernoulli_doob}.}
    \label{fig:bernoulli_doob}
\end{figure}
\Cref{prop:hard-deadline-np} shows that for the simple-vs-simple the hard-deadline problem can be tackled directly at the e-process level, without solving the Bellman recursion. The optimal level-$\alpha$ sequential test rejects the null as soon as $U^{\star,\infty}$ crosses  $\alpha^{-1}$.  For $(\A,\phi)=(\A^{\pi_0},\phi^{\pi_0})$,  $U^{\star,\infty}$ is the wealth of any predictable policy satisfying $\nu_t^\star(x^{t-1})=\pi_0^{\otimes}(x_t\in\cdot\mid x^{t-1},\Gamma^\star)$ for $t\leq T$. This canonical policy implementation need not be Markovian, as it may depend on the full history. This is consistent with \Cref{thm:markovian_reduction}, which does not require every optimiser to be Markovian. Concrete examples are discussed in \Cref{sec:examples} and \Cref{sec:app-examples}, where the canonical induced policy is Markovian for Gaussian mean testing, not necessarily so in the Bernoulli case.

\section{Exponentially decaying reward}
\label{sec:exponential-reward}
We now specialise to the exponentially decaying reward $R(t)=e^{-t/T}$, with $T>0$ the time scale. It is convenient to factor out the explicit time dependence from the value function and rescale the wealth. Thus, we let $z=\alpha w$ and $V_t(z/\alpha)=e^{-t/T}v(z)$. The Bellman recursion \eqref{eq:bellman} now reads
$$v(z)=e^{-1/T}\sup_{a\in\A}\int_\X\Bigl(\mathbf{1}_{\{z\phi(a,x)\geq 1\}}+v(z\phi(a,x))\mathbf{1}_{\{z\phi(a,x)<1\}}\Bigr)\dd\pi_1(x)\,.$$
The problem is therefore stationary ($t$ no longer appears) and the only remaining inhomogeneity is the rejection boundary at $z=1$. In general, this stationary Bellman equation does not admit a closed-form solution, but can be tackled numerically by standard dynamic-programming methods, after discretising wealth and actions, or by truncating the horizon and applying backward induction (as for \cref{fig:heatmaps-Gaussian}). However, our focus here is different. We derive an analytically tractable approximation which, like GRO, selects a constant betting rule and admits an explicit large-threshold interpretation.

Since the test starts from $z=\alpha$, typically much smaller than $1$, it is natural to start by studying the behaviour far from the threshold. In that regime, the truncation at $1$ should have only a limited effect, suggesting the approximate homogeneous equation
\begin{equation}\label{eq:hom}v(z)\approx e^{-1/T}\sup_{a\in\A}\int_\X v(z\phi(a,x))\dd\pi_1(x)\,.\end{equation}
Under this approximation, the homogeneous equation becomes scale invariant in the wealth, suggesting the multiplicative ansatz $v(z)=cz^\eta$, with $c>0$ and $\eta>0$. Substituting this into \eqref{eq:hom} factors out the dependence on $z$. For the ansatz to yield a solution, there must exist $a_\star\in\A$ and $\eta_\star>0$ such that
\begin{equation}\label{eq:cond}\sup_{a\in\A}\E_{\pi_1}\bigl[\phi(a,\,\cdot\,)^{\eta_\star}\bigr] = \E_{\pi_1}\bigl[\phi(a_\star,\,\cdot\,)^{\eta_\star}\bigr] = e^{1/T}\,.\end{equation}
When this holds, $z\mapsto cz^{\eta_\star}$ is a fixed point of the homogeneous Bellman equation \eqref{eq:hom} and we define the exponential-decay-optimal (EDO) criterion: play the action $a_\star$ from \Cref{eq:cond}, independently of wealth. For $T$ large enough, the existence of a pair satisfying \Cref{eq:cond} is natural when the action space is sufficiently rich: if $e^{1/T}$ is close to $1$, it is enough to have actions with a small amount of moment growth under $\pi_1$. This becomes explicit in \Cref{prop:exp-reward-summary} for the simple-vs-simple case.

The threshold boundary effect has deliberately been ignored. Near $z=1$, the stationary action should only be viewed as a baseline: once some outcomes already trigger rejection, extra payoff beyond the threshold has no value, and a wealth-dependent action may improve performance by reallocating betting mass to non-rejection outcomes. For instance, in Bernoulli testing, if observing $1$ already rejects, one can cap the payoff assigned to $1$ at the level needed to cross the threshold and shift the remaining mass to increase the payoff assigned to $0$. Such boundary-aware corrections can be applied to baselines such as EDO or GRO, and are related in spirit to \cite{fischer2026improving}. Numerically, in \cref{fig:edo-approx}, EDO already provides a close approximation to the Bellman optimum, and the boundary-aware capped variant nearly matches it.

\begin{figure}[!t]
    \centering
    \includegraphics[width=.75\linewidth]{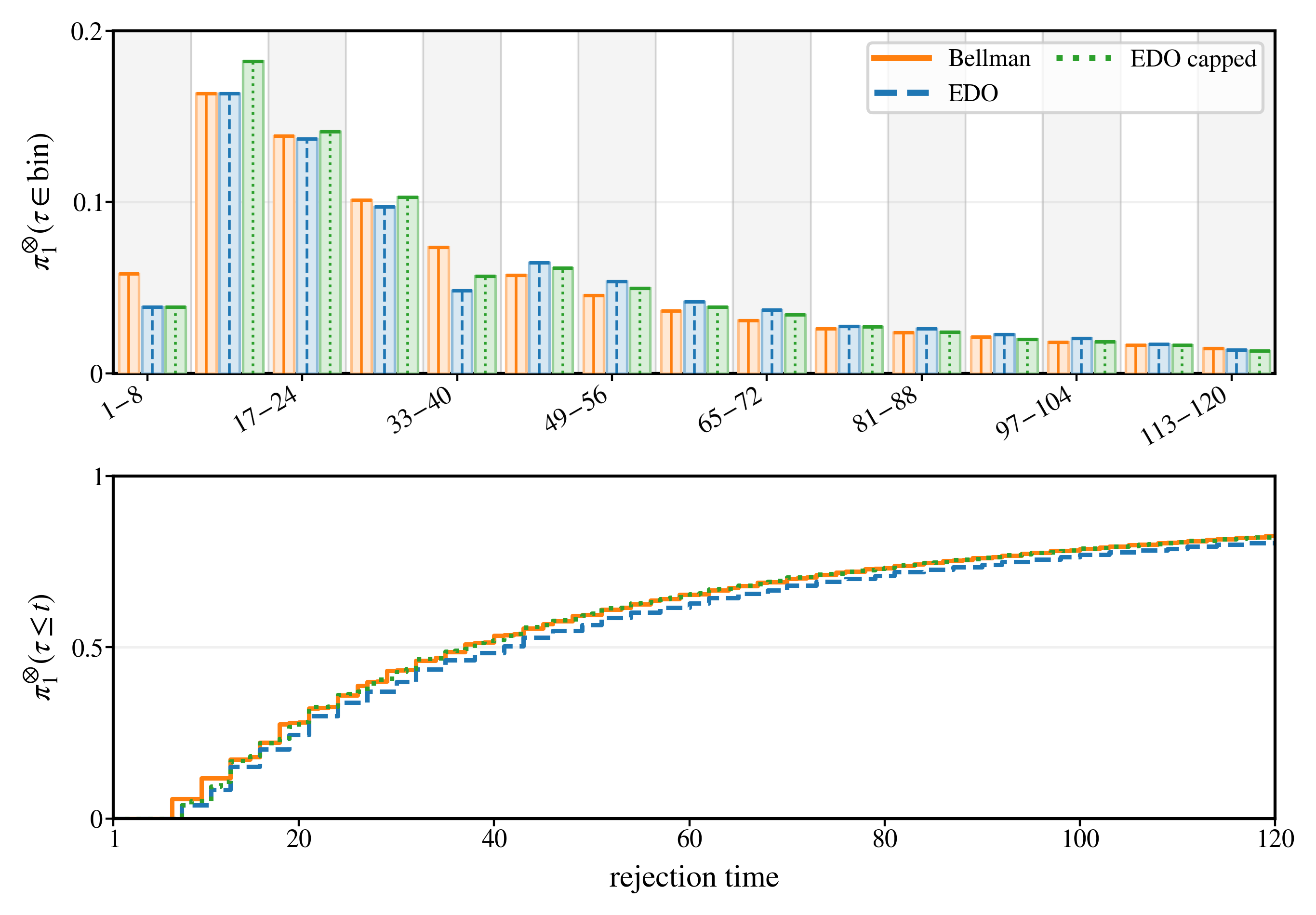}
        \caption{Comparison of EDO with the numerical Bellman optimum for $\pi_0=\Bernoulli(1/2)$ vs $\pi_1=\Bernoulli(2/3)$ in terms of rejection times. With boundary-aware capping, the approximation is nearly indistinguishable. Details in \Cref{sec:appendix-fig-edo-approx}.}
        \label{fig:edo-approx}
\end{figure}

The following  bound shows that the policy suggested by the homogeneous equation has the right large-threshold order. We write $\tau^a$ for the rejection time under the policy that always plays $a\in\A$.
\begin{proposition}
\label{prop:approximation-constant-strategy}
Let $(a_\star,\eta_\star)$ satisfy \eqref{eq:cond}, with $\eta_\star\in(0,1)$ and $\phi(a_\star,\,\cdot\,)\leq \Phi$, $\pi_1$-almost surely. Then
$$(\alpha/\Phi)^{\eta_\star}\leq \E_{\pi_1^\otimes}[e^{-\tau^{a_\star}/T}] \leq \sup_{a^\infty\text{ Markov}}\E_{\pi_1^\otimes}[e^{-\tau^{a^\infty}/T}]\leq \alpha^{\eta_\star}\,.$$
\end{proposition}
Thus EDO provides a first-order proxy: it identifies a constant policy with optimal $\alpha$-exponent, in the sense that both the optimal Markov value and the value achieved by the constant action $a_\star$ have logarithm $\eta_\star\log(\alpha)+O(1)$ as $\alpha\to 0$. Hence, like GRO, EDO gives an $\alpha$-agnostic constant-bet criterion whose optimality is first-order in the rejection threshold \citep{agrawal2025stopping}.

To make things more concrete, we consider the simple-vs-simple case, with $\pi_1\ll\pi_0$ and $(\A, \phi)=(\PR^{\pi_0},\phi^{\pi_0})$. Then, for $T$ large enough a closed-solution of \eqref{eq:cond} typically exists: $a^\star$ is a likelihood-ratio tilt of $\pi_0$, and $\eta^\star$ is expressed through a Rényi divergence. For $\eta\in(0,1)$, let $L_\eta=(\dd\pi_1/\dd\pi_0)^{1/(1-\eta)}$ and $D_\xi(\pi\|\pi') = \frac{1}{\xi - 1} \log \E_{\pi'}[(\dd\pi/\dd\pi')^\xi]$ denote the Rényi divergence of order $\xi>1$.

\begin{proposition}
\label{prop:exp-reward-summary}
Let $G = \{(1-\eta)\log\E_{\pi_0}[L_\eta]\,:\,\eta\in[0,1)\,,\,\E_{\pi_0}[L_\eta]<\infty\}$. If $1/T\in(0,\sup G)$, then there exist $\eta_\star\in(0,1)$ and $\pi_\star\ll\pi_0$ satisfying \eqref{eq:cond}. Moreover,
$\dd\pi_\star/\dd\pi_0=L_{\eta_\star}/\E_{\pi_0}[L_{\eta_\star}]$
and $\log(\sup_{\pi\ll\pi_0}\E_{\pi_1}[(\dd\pi/\dd\pi_0)^{\eta_\star}]) = \eta^\star D_{1/(1-\eta_\star)}(\pi_1\|\pi_0) = 1/T$.
\end{proposition}
Under the mild integrability condition $G\neq\{0\}$, \eqref{eq:cond} admits a solution for all sufficiently large $T$. EDO then selects the e-variable $\dd\pi_\star/\dd\pi_0=L_{\eta_\star}/\E_{\pi_0}[L_{\eta_\star}]$, a power tilt of $\pi_0$ by the likelihood ratio, with $\eta_\star$ chosen so that $\eta_\star D_{1/(1-\eta_\star)}(\pi_1\|\pi_0)=1/T$. We can compare EDO with the GRO criterion. In the simple-vs-simple case, GRO recovers the constant action $\pi_1$,  achieving $\sup_{\pi\ll\pi_0}\E_{\pi_1}\bigl[\log(\dd\pi/\dd\pi_0)\bigr]=\KL{\pi_1}{\pi_0}$. This KL divergence is also the quantity governing the first-order expected stopping-time scale under GRO: for the likelihood-ratio strategy, $\E_{\pi_1^\otimes}[\tau^{\pi_1}]\sim\log(1/\alpha)/\KL{\pi_1}{\pi_0}$ as $\alpha\to0$ \citep{agrawal2025stopping}. By contrast, EDO plays $\pi^\star$ every round. In this case, the  large-threshold relation reads $-T\log\E_{\pi_1^\otimes}[e^{-\tau^{\pi_\star}/T}]\sim\log(1/\alpha)/D_{1/(1-\eta_\star)}(\pi_1\|\pi_0)$. The two pictures are consistent, as  $T\to\infty$ the objective approaches the linear regime $e^{-\tau/T}\sim 1-\tau/T$, while $\eta_\star\to0$ and $D_{1/(1-\eta_\star)}(\pi_1\|\pi_0)\to\KL{\pi_1}{\pi_0}$ (cf.~\cref{fig:edo-gro}).

There is another comparison with GRO worth making. GRO favours actions with large expected log-growth under the alternative, and positive log-growth leads to almost sure eventual rejection. Conversely, the exponentially decaying reward values earlier rejection, and for small time scales may select an action with non-positive log-drift. Thus early rejection can come at the cost of power. The next result records this trade-off for constant policies in the general action-space  formulation.

\begin{figure}[!t]
    \centering
    \includegraphics[width=.75\linewidth]{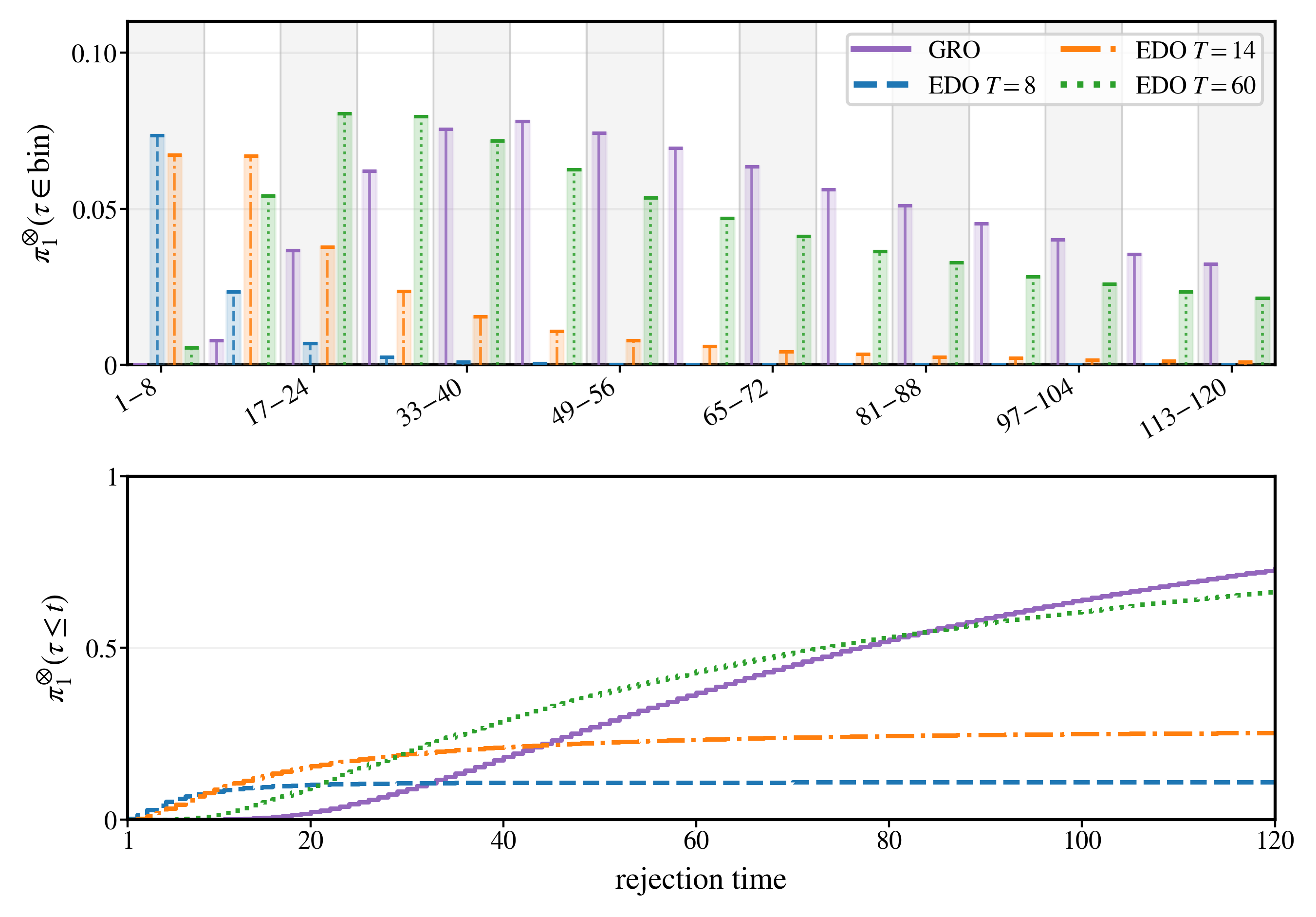}
        \caption{Effect of the reward time scale on EDO rejection-time distributions for $\pi_0=\N(0,1)$ vs $\pi_1=\N(.25,1)$: smaller time scales favour earlier rejection, with a visible power trade-off. GRO is shown for comparison. Details in \Cref{sec:appendix-fig-edo-gro}.}
        \label{fig:edo-gro}
\end{figure}

\begin{proposition}
\label{prop:power-constant-strategy}
Fix $a\in\A$, assume that $\phi(a,\,\cdot\,)>0$, $\pi_1$-almost surely, and that $\gamma(a)=\E_{\pi_1}[\log\phi(a,\,\cdot\,)]\in[-\infty,\infty]$ is well-defined. Then, $\pi_1^\otimes(\tau_w^a<\infty)=1$ for every $w\in(0,\alpha^{-1})$ if and only if either $\gamma(a)>0$, or $\gamma(a)=0$ and $\pi_1(\log\phi(a,\,\cdot\,)>0)>0$. Moreover, when $\gamma(a)<0$, if $\phi(a,\,\cdot\,)\leq\Phi$, $\pi_1$-almost surely, and there exists $\kappa>0$ such that $\E_{\pi_1}[\phi(a,\,\cdot\,)^\kappa]=1$, then $(\alpha w/\Phi)^\kappa\leq\pi_1^\otimes(\tau_w^a<\infty)\leq(\alpha w)^\kappa$ for every $w\in(0,\alpha^{-1})$.
\end{proposition}
\begin{figure}[!b]
    \centering
    \includegraphics[width=0.5\textwidth]{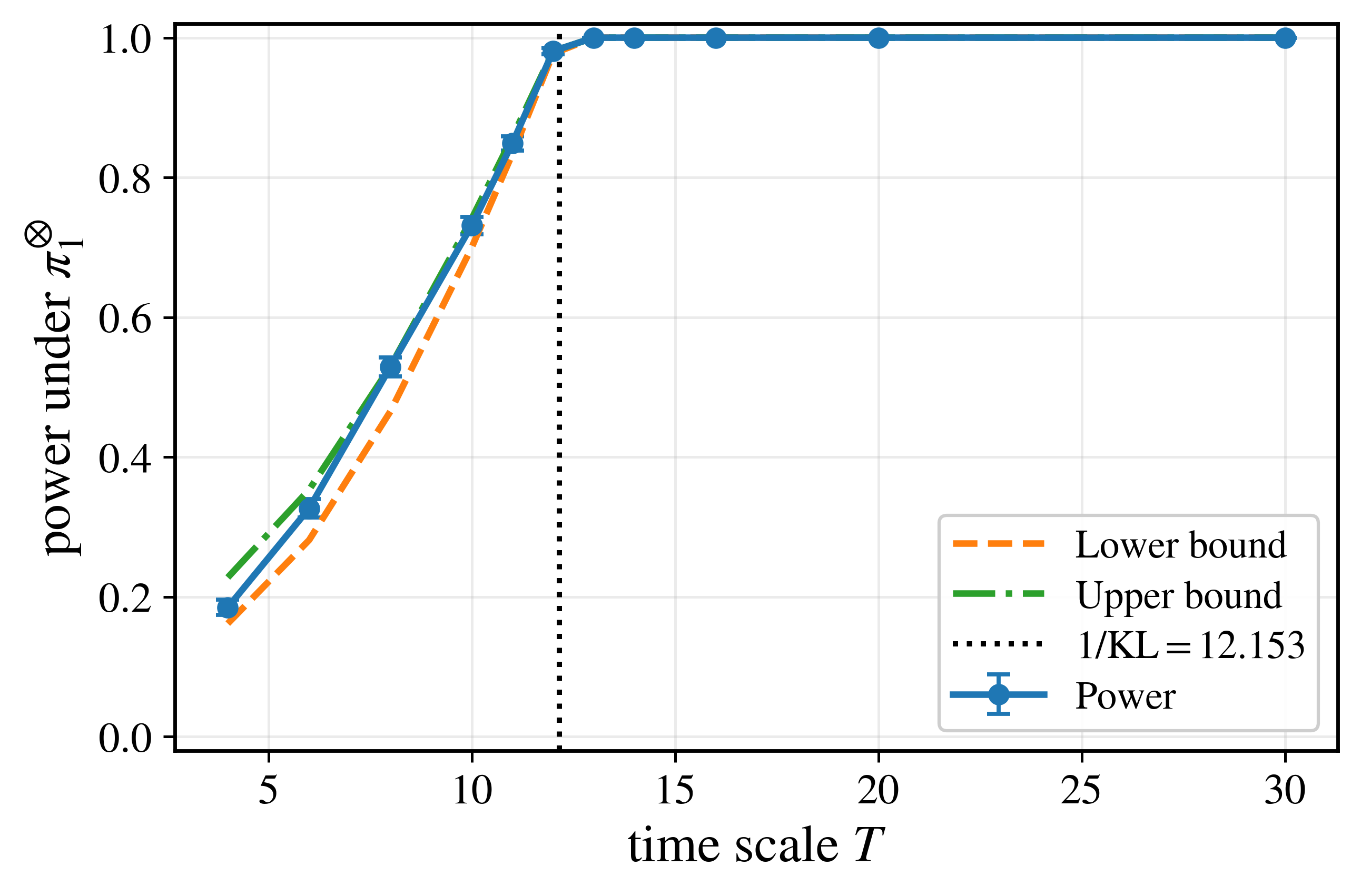}
    \caption{Power trade-off in $\pi_0=\Bernoulli(.5)$ vs $\pi_1=\Bernoulli(.7)$ for  EDO criterion. Details in \Cref{sec:appendix-fig-power-bound-check}.}
    \label{fig:power-bound-check}
\end{figure}

\Cref{prop:power-constant-strategy} essentially says that a strategy betting the same action each round has power $1$ against $\pi_1^\otimes$ if and only if its log-wealth has non-negative drift under $\pi_1$. In the simple-vs-simple setting, under the exponentially decaying reward this criterion becomes especially explicit for the constant action $\pi_\star$ selected by \Cref{prop:approximation-constant-strategy}, for which we get $\gamma(\pi_\star)=(\KL{\pi_1}{\pi_0}-1/T)/(1-\eta_\star)$.
In such case for a test to have power $1$ we need $\KL{\pi_1}{\pi_0}$ to be at least $1/T$. The trade-off between $T$ and the power is illustrated in \cref{fig:power-bound-check} for two Bernoulli distributions.

\section{Two illustrative examples}\label{sec:examples}
We discuss two simple-vs-simple examples. The full calculations are deferred to \Cref{sec:app-examples}.
\paragraph{Testing Gaussians.} We consider the Gaussian case, where $\pi_0=\N(0,\sigma^2)$ and $\pi_1=\N(\mu,\sigma^2)$, with $\mu>0$ and $\sigma>0$. We fix $(\A,\phi) = (\PR^{\pi_0},\phi^{\pi_0})$. For the hard-deadline reward $R(t)=\one_{t\leq T}$, \Cref{prop:hard-deadline-np} gives the event $\Gamma^\star=\{S_T\geq c\}$, where $c=\sigma\sqrt T F^{-1}(1-\alpha)$, $S_T=\sum_{t=1}^T x_t$ and $F$ is the cdf of a standard normal. The optimal e-process is the Doob martingale $$U_t^\star=\alpha^{-1}\pi_0^\otimes(S_T\geq c\mid S_t)=\frac{1}{\alpha}\left(1-F\left(\frac{c-S_t}{\sigma\sqrt{T-t}}\right)\right)\,,$$ for $t<T$, with terminal convention $U_T^\star=\alpha^{-1}\one_{[c,\infty)}(S_T)$. The induced strategy is Markovian, since the same conditional probability is strictly increasing in $S_t$, hence in the current wealth (see \Cref{sec:app-gauss}). Note that $S_T$ given $S_t$ is Gaussian with full support, so the conditional probability above is strictly below one for every $t<T$. Thus the exact Gaussian hard-deadline test never rejects before the deadline: it rejects at $T$ precisely on $\Gamma^\star$, and hence effectively coincides with the fixed-sample Neyman--Pearson test. This contrasts with \cref{fig:heatmaps-Gaussian}, where the action space is restricted to Gaussian-shift bets, so the displayed policy is the Bellman optimum within that smaller class rather than the unrestricted Doob strategy. Finally, the magnitude of $\mu$ affects the power but not the rejection region $\Gamma^\star$, only its sign does.

For the exponentially decaying reward $R(t)=e^{-t/T}$, the EDO action from \Cref{prop:exp-reward-summary} is obtained solving $(T\eta^\star)^{-1} = D_\xi(\pi_1\|\pi_0)=\xi\mu^2/(2\sigma^2)^{-1}$, which gives $\eta^\star = 2\sigma^2/(\mu^2T + 2\sigma^2)\in(0,1)$. The EDO action is $\pi_\star=\N\left(2\sigma^2/(\mu T)+\mu,\,\sigma^2\right)$, corresponding to the e-variable
$$E_\star(x) = \frac{\dd\pi_\star}{\dd\pi_0}(x)=\exp\left(\left(\frac{\mu}{\sigma^2}+\frac{2}{\mu T}\right)(x-\mu)-\frac{2\sigma^2}{\mu^2T^2}\right)$$
in the testing-by-betting game.
Compared with the GRO strategy, which plays $\pi_1=\N(\mu,\sigma^2)$, EDO is more aggressive, shifting further away from the null. As $T\to\infty$,  $\pi_\star$ recovers the GRO  $\pi_1$. By \Cref{prop:power-constant-strategy}, this  aggressiveness preserves power one if and only if $T\geq\KL{\pi_1}{\pi_0}^{-1}=2\sigma^2/\mu^2$.

\paragraph{Testing Bernoullis.} Consider the case $\pi_0=\Bernoulli(p_0)$ and $\pi_1=\Bernoulli(p_1)$, with $0<p_0<p_1<1$. Set $(\A,\phi)=(\A^{\pi_0},\phi^{\pi_0})$. For the hard-deadline  $R(t)=\one_{t\leq T}$, identifying  $\Gamma^\star$ is  more delicate than in the Gaussian case, since  Bernoulli  is atomic. This can still be done exactly through a finite integer programme (see \Cref{sec:app-bern}). However, the canonical Doob strategy may fail to be Markovian in the wealth (cf.~the discussion after \Cref{prop:hard-deadline-np}). For large $T$, a simpler and  explicit upper-tail approximation holds, only depending on $S_T = \sum_{t=1}^T x_t$ (see \Cref{sec:app-bern}).

For  $R(t)=e^{-t/T}$, to determine the EDO action, $\eta_\star$ does not admit a closed-form but can be easily computed numerically solving a one-dimensional equation. The resulting action $\pi_\star=\Bernoulli(p_\star)$ has a closed-form expression in terms of $\eta_\star$ (see \Cref{sec:app-bern}), from which it follows that EDO is more aggressive than GRO: $p_\star>p_1$. As $T\to\infty$, we have $\eta_\star\to0$, and $\pi_\star$ recovers the GRO action $\pi_1$. By \Cref{prop:power-constant-strategy}, EDO has power one if and only if $T\geq\KL{\Bernoulli(p_1)}{\Bernoulli(p_0)}^{-1}$ (cf.~\cref{fig:power-bound-check}).

\section{Discussion}

We close by discussing some limitations of the present formulation and possible open directions. First, the current work only focuses on the case of a simple alternative $\pi_1$. This choice makes the control problem transparent, but it is restrictive when the alternative is composite. A natural extension would be to replace the expectation under $\pi_1^\otimes$ by a Bayesian, worst-case, or relative worst-case criterion over a class of alternatives, in the spirit of the GROW and REGROW criteria from safe testing \citep{grunwald2024safe}. The resulting problem would be a robust version of our time-sensitive control formulation. 

A related open direction concerns composite nulls. Our explicit approximation results are derived in the simple-vs-simple setting, while under GRO the extension to a composite null often admits a least-favourable-null interpretation through a KL projection of the alternative onto the null model. It would be interesting to understand whether a similar projection principle exists for the time-sensitive objectives considered here. 

Another point concerns the role of the significance level. Throughout our analysis, we fix a  significance level $\alpha$, making the objective transparent: the reward is assigned to the hitting time of $\alpha^{-1}$. However, the same wealth process determines hitting times for all rejection boundaries, and one may want a strategy that performs well at several values of $\alpha$ simultaneously. A natural extension would be to replace the single-threshold objective by an aggregate, worst-case, or relative worst-case criterion over a family of significance levels, in analogy with the composite-alternative criteria discussed above. In any case, for the exponential decay reward, EDO is already $\alpha$-agnostic. 

Finally,  our analysis is restricted to non-negative rewards. This is natural when $R(t)$ represents the \emph{value} of a rejection at $t$, and it allows to invoke a result of \cite{schal1987stationary} to prove the Markovian reduction. However, it excludes unbounded loss-type objective, such as directly minimising $\E_{\pi_1^\otimes}[\tau]$ (which would correspond to $R(t) = -t$). Extending the theory to suitable unbounded rewards remains an interesting direction. However, the exponentially decaying reward already gives a \emph{proxy} for minimising the expected rejection time: as $T\to\infty$, we enter the linear regime $e^{-t/T}\approx 1-t/T$ and EDO recovers  GRO.

\paragraph{LLM usage.} The authors acknowledge the use of LLMs to improve exposition, explore proof ideas, and generate code. The authors take full responsibility for the content of the paper.

\newpage

\bibliographystyle{abbrvnat}
\bibliography{bibliography}

\newpage
\appendix

\crefalias{section}{appendix}
\crefalias{subsection}{appendix}
\crefalias{subsubsection}{appendix}
\crefname{appendix}{Appendix}{Appendices}
\Crefname{appendix}{Appendix}{Appendices}

\section{Notation and conventions}
\label{app:notation}

We use the following notation throughout. We write $\Nn=\{0,1,2,\dots\}$ and $\R_+=[0,\infty)$. For a statement $B$, $\one_B$ denotes its indicator. We use the convention that empty products are equal to one and that $\inf\emptyset=\infty$. All logarithms are natural logarithms.

For a sequence $x_1,x_2,\dots$, we write $x^t=(x_1,\dots,x_t)$ for the first $t$ coordinates, with $x^0$ denoting the empty history. We write $x^\infty=(x_1,x_2,\dots)$ for the full infinite sequence.

All measurable spaces are assumed to be standard Borel unless explicitly stated otherwise. A standard Borel space is a measurable space that is measurably isomorphic to a Borel subset of a Polish space. This convention is used to avoid repeated measurability qualifications and to ensure that regular conditional probabilities and measurable kernels are available.

If $\X$ is a measurable space, $\PR_\X$ denotes the set of Borel probability measures on $\X$. For $\pi\in\PR_\X$, $\E_\pi$ denotes expectation under $\pi$. If $\pi$ is a one-step law, then $\pi^\otimes$ denotes the infinite product measure on $\X^\infty$, corresponding to an i.i.d. sequence with common law $\pi$.

For probability measures $\nu$ and $\pi$ on the same measurable space, $\nu\ll\pi$ means that $\nu$ is absolutely continuous with respect to $\pi$, and $\dd\nu/\dd\pi$ denotes a Radon--Nikodym derivative. Statements involving such derivatives are understood up to almost-sure equivalence unless a particular version is fixed.

Unless otherwise specified, filtrations are the natural filtrations generated by the observed sequence. A random time is called a stopping time with respect to this filtration. All maps and kernels are taken to be Borel measurable whenever this is required by the context.

Measurability is always meant with respect to the Borel $\sigma$-fields determined by the spaces under consideration. Maps, policies, kernels, and selections are taken to be Borel measurable unless explicitly stated otherwise.

\newpage
\section{Omitted proofs}\label{app:proofs}
\subsection[Proof of Theorem 1]{Proof of \Cref{thm:markovian_reduction}}

We prove here the Markovian reduction stated in \Cref{thm:markovian_reduction}:
\begin{theorem*}[Markovian reduction]
For any reward function $R$ we have $$\sup_{A^\infty}\E_{\pi_1^\otimes}[R(\tau^{A^\infty})]=\sup_{a^\infty\text{ Markov}}\E_{\pi_1^\otimes}[R(\tau^{a^\infty})]\,.$$
\end{theorem*}

First, we define a \emph{reduced} control model. This is a positive Borel control model whose state records the current time and wealth, until the threshold is reached. We set
$$\Sc=\big(\Nn\times[0,\alpha^{-1})\big)\cup\{\circledast\}\,,$$
where $\circledast$ is a terminal state. The interpretation is that $(t,w)\in\Nn\times[0,\alpha^{-1})$ means that $t$ observations have been seen, the current wealth is $w$, and the threshold has not yet been reached. The state $\circledast$ is absorbing and represents the threshold having been reached.

Define $\Phi:\Sc\times\A\times\X\to\Sc$ by
$$\Phi((t,w),a,x)=\begin{cases}\circledast&\text{if }w\phi(a,x)\geq\alpha^{-1}\,,\\(t+1,w\phi(a,x))&\text{if }w\phi(a,x)<\alpha^{-1}\,,\end{cases}$$
for $(t,w)\in\Nn\times[0,\alpha^{-1})$, and $\Phi(\circledast,a,x)=\circledast$. Since $\phi$ is jointly Borel, $\Phi$ is Borel. The alternative $\pi_1$ induces a transition kernel $q$, defined as
$$q(B\mid s,a)=\pi_1(\{x\in\X:\Phi(s,a,x)\in B\})\,,$$
for any Borel $B\subseteq\Sc$. This is a Borel measurable stochastic kernel on $\Sc$, by the Borelness of $\Phi$. We now define the one-step reward $r:\Sc\times\A\times\Sc\to\R_+$ as
$$r((t,w),a,s')=R(t+1)\one_{\{s'=\circledast\}}$$
for $(t,w)\in\Nn\times[0,\alpha^{-1})$, and $r(\circledast,a,s')=0$ for every $a\in\A$ and $s'\in\Sc$. A reward is received exactly when the process first enters $\circledast$, namely when the testing-by-betting procedure rejects.

A \emph{reduced policy} is a sequence $S^\infty=(S_0,S_1,\dots)$, where, for each $n\geq0$, $S_n$ is a measurable kernel from $\Sc^{n+1}$ to $\A$. This means that for each $(s_0,\dots,s_n)\in\Sc^{n+1}$, $S_n(\cdot\mid s_0,\dots,s_n)$ is a Borel probability measure on $\A$, and the map $(s_0,\dots,s_n)\mapsto S_n(B\mid s_0,\dots,s_n)$ is Borel for every Borel set $B\subseteq\A$. Note that the reduced model allows arbitrary state-history-dependent randomised policies, and deterministic reduced policies are included as special cases, when each $S_n$ is a Dirac mass. A \emph{stationary} reduced policy is a Borel map $f:\Sc\to\A$.

For a given reduced policy $S^\infty$ and an initial state $s\in\Sc$, let $I(S^\infty,s)$ be its expected total reward. We define
$$v(s)=\sup_{S^\infty}I(S^\infty,s)\,,$$
where the supremum is over all reduced policies.

The reduced model defined above is a positive Borel dynamic-programming model. The following is the only non-elementary ingredient we will need to prove \Cref{thm:markovian_reduction}. This is based on \cite[Corollary~1.1]{schal1987stationary}.

\begin{lemma}
\label{lemma:stationary_reduced_model}
For every initial state $s\in\Sc$,
$$v(s)=\sup_{f\text{ stationary}}I(f,s)\,,$$
where the supremum on the right-hand side is over all Borel stationary policies $f:\Sc\to\A$.
\end{lemma}

\begin{proof}
Let us first consider the larger class of full-history reduced policies. Such a policy is a sequence $\bar S^\infty=(\bar S_0,\bar S_1,\dots)$, where, for each $n\geq0$, $\bar S_n$ is a Borel measurable kernel from $\Sc\times(\A\times\Sc)^n$ to $\A$. Thus $\bar S_n$ may depend on the full reduced history $(s_0,a_0,s_1,\dots,a_{n-1},s_n)$. Let $\bar v(s)=\sup_{\bar S^\infty}I(\bar S^\infty,s)$, where the supremum is over all full-history reduced policies. Then $\sup_{f\text{ stationary}}I(f,s)\leq v(s)\leq\bar v(s)$, because stationary policies are included in the reduced policies used to define $v$, and every such reduced policy can be viewed as a full-history reduced policy which ignores the past actions.

We now apply \cite[Corollary~1.1]{schal1987stationary} to the model with full-history policies. Indeed, $\Sc$ and $\A$ are standard Borel spaces, the feasible action set is constant and equal to $\A$, the transition kernel $q$ is Borel measurable, and the one-step reward $r$ is non-negative and Borel measurable. Since the process can receive a non-zero reward at most once and $R$ is bounded, $\bar v(s)<\infty$ for every $s\in\Sc$. Hence the result applies. In the notation of \cite{schal1987stationary}, our $\bar v$ is the value function of the positive model and Corollary~1.1 says that stationary policies are persistently $\varepsilon$-optimal. Applied to the initial distribution concentrated at $s$, this means that for every $\varepsilon>0$ there exists a stationary policy $f_\varepsilon:\Sc\to\A$ such that $I(f_\varepsilon,s)\geq(1-\varepsilon)\bar v(s)$, with the case $\bar v(s)=0$ being trivial by non-negativity. Therefore $\bar v(s)=\sup_{f\text{ stationary}}I(f,s)$.
Combining this equality with the previous inequalities yields $v(s)=\sup_{f\text{ stationary}}I(f,s)$.
\end{proof}

\begin{proof}[Proof of the Markovian reduction]
We compare the original problem with the reduced control model. First, fix an arbitrary policy $A^\infty=(A_1,A_2,\dots)$ for the original problem. Let $s_0=(0,1)$ and define recursively
$$\widehat S_0^{A^\infty}(x^0)=s_0$$
and
$$\widehat S_{t+1}^{A^\infty}(x^{t+1})=\Phi(\widehat S_t^{A^\infty}(x^t),A_{t+1}(x^t),x_{t+1})\,.$$
These are all Borel maps. By induction over $t$, for any $x^\infty\in\X^\infty$,
$$\widehat S_t^{A^\infty}(x^t)=\begin{cases}(t,\widehat W_t^{A^\infty}(x^t))&\text{if }\tau^{A^\infty}(x^\infty)>t\,,\\\circledast&\text{if }\tau^{A^\infty}(x^\infty)\leq t\,.\end{cases}$$
Now, for any $x^\infty\in\X^\infty$,
$$\sum_{t=0}^\infty r\left(\widehat S_t^{A^\infty}(x^t),A_{t+1}(x^t),\widehat S_{t+1}^{A^\infty}(x^{t+1})\right)=R(\tau^{A^\infty}(x^\infty))\,.$$
Indeed, if $\tau^{A^\infty}(x^\infty)=T<\infty$, then the path enters $\circledast$ at step $T$, and the only non-zero term in the sum is for $t=T-1$, where it is equal to $R(T)$. On the other hand, if $\tau^{A^\infty}(x^\infty)=\infty$, then the path never enters $\circledast$ and the sum is $0=R(\infty)$.

Now, let us construct a reduced policy starting from $A^\infty$. For each $n\geq0$, consider the joint law under $X^\infty\sim\pi_1^\otimes$ of
$$\big(\widehat S_0^{A^\infty}(X^0),\dots,\widehat S_n^{A^\infty}(X^n),A_{n+1}(X^n)\big)\,.$$
Since $\Sc$ and $\A$ are standard Borel spaces, there exists a regular conditional distribution of the last coordinate given the first $n+1$ ones. Let $S_n^{A^\infty}$ be such a kernel, that is,
$$S_n^{A^\infty}(B\mid s_0,\dots,s_n)=\pi_1^\otimes\big(A_{n+1}(X^n)\in B\mid \widehat S_0^{A^\infty}=s_0,\dots,\widehat S_n^{A^\infty}=s_n\big)$$
for every Borel set $B\subseteq\A$. We claim that the state process generated by $(S_n^{A^\infty})_{n\geq0}$ under the transition kernel $q$ has the same law as the process $(\widehat S_t^{A^\infty}(X^t))_{t\geq0}$
generated by $A^\infty$ in the original model, under $X^\infty\sim\pi_1^\otimes$. Indeed, by construction the two processes start from the same state. Moreover, conditionally on $\widehat S_0^{A^\infty},\dots,\widehat S_n^{A^\infty}$, the conditional law of $A_{n+1}(X^n)$ is $S_n^{A^\infty}(\cdot\mid \widehat S_0^{A^\infty},\dots,\widehat S_n^{A^\infty})$, and $X_{n+1}$ is independent of $X^n$ with law $\pi_1$. Hence, for every Borel set $C\subseteq\Sc$,
$$\pi_1^\otimes\big(\widehat S_{n+1}^{A^\infty}\in C\mid \widehat S_0^{A^\infty},\dots,\widehat S_n^{A^\infty}\big)=\int_\A q(C\mid \widehat S_n^{A^\infty},a)\,\dd S_n^{A^\infty}(a\mid \widehat S_0^{A^\infty},\dots,\widehat S_n^{A^\infty})\,,$$
which is precisely the transition rule in the reduced model. It follows by induction on $n$ that, for every $n\geq0$, the law of the first $n+1$ states in the reduced model, started from $(0,1)$, controlled by $S^{A^\infty,\infty}$, and evolved according to $q$, is the same as the law of $(\widehat S_0^{A^\infty},\dots,\widehat S_n^{A^\infty})$ under $X^\infty\sim\pi_1^\otimes$.

Let $\widehat\tau^{A^\infty}=\inf\{t\geq1:\widehat S_t^{A^\infty}(X^t)=\circledast\}$. Since $\widehat S_t^{A^\infty}(x^t)=\circledast$ if and only if $\tau^{A^\infty}(x^\infty)\leq t$, we have $\widehat\tau^{A^\infty}=\tau^{A^\infty}$. Since the total reward in the reduced model is equal to $R$ evaluated at the first entrance time into $\circledast$, the equality in law just proved gives
$$I(S^{A^\infty,\infty},(0,1))=\E_{\pi_1^\otimes}[R(\widehat\tau^{A^\infty})]=\E_{\pi_1^\otimes}[R(\tau^{A^\infty})]\,.$$
Taking the supremum over original policies, we obtain
\begin{equation}\label{eq:ineq}\sup_{A^\infty}\E_{\pi_1^\otimes}[R(\tau^{A^\infty})]\leq\sup_{S^\infty}I(S^\infty,(0,1))=v(0,1)\,.\end{equation}
Now, we show that every stationary reduced policy gives rise to a Markov policy in the original model achieving the same expected total reward. Fix a stationary reduced policy $f:\Sc\to\A$. Define $a^{f,\infty}=(a_t^f)_{t\geq0}$ by
$$a_t^f(w)=f(t,w)$$
for $t\geq0$ and $w\in[0,\alpha^{-1})$. This is a Markov policy. Let $A^{f,\infty}$ be the corresponding ordinary policy, defined recursively by
$$A^{f,\infty}_{t+1}(x^t)=\begin{cases}a_t^f(\widehat W_t^{A^{f,\infty}}(x^t))&\text{if }\widehat W_t^{A^{f,\infty}}(x^t)<\alpha^{-1}\,,\\f(\circledast)&\text{if }\widehat W_t^{A^{f,\infty}}(x^t)\geq\alpha^{-1}\,.\end{cases}$$

Under $X^\infty\sim\pi_1^\otimes$, the state process induced by $A^{f,\infty}$ starts from $(0,1)$ and, before entering $\circledast$, moves from $(t,w)$ to $\circledast$ if $w\phi(a_t^f(w),X_{t+1})\geq\alpha^{-1}$, and to $(t+1,w\phi(a_t^f(w),X_{t+1}))$ otherwise. This is exactly the transition rule of the reduced model under the stationary policy $f$. Hence, by induction on $n$, the law of the first $n+1$ induced states is the same as the law of the first $n+1$ states in the reduced model, started from $(0,1)$, controlled by $f$, and evolved according to $q$. Therefore,
$$\E_{\pi_1^\otimes}[R(\tau^{A^{f,\infty}})]=I(f,(0,1))\,.$$
Taking the supremum over stationary policies $f$, and using \Cref{lemma:stationary_reduced_model}, we obtain
$$\sup_{a^\infty\text{ Markov}}\E_{\pi_1^\otimes}[R(\tau^{a^\infty})]\geq\sup_{f\text{ stationary}}I(f,(0,1))=v(0,1)\,.$$
Together with \eqref{eq:ineq}, this gives
$$\sup_{a^\infty\text{ Markov}}\E_{\pi_1^\otimes}[R(\tau^{a^\infty})]\geq v(0,1)\geq\sup_{A^\infty}\E_{\pi_1^\otimes}[R(\tau^{A^\infty})]\,.$$
The reverse inequality follows because Markov policies induce a subclass of admissible ordinary policies in the original problem. This proves the claim.
\end{proof}

\subsection[Proof of Theorem 2]{Proof of \Cref{thm:eprocess_equivalence}}
We now prove \Cref{thm:eprocess_equivalence}, which we restate here:
\begin{theorem*}[Simple-vs-simple e-process equivalence]
Fix $\pi_0,\pi_1\in\PR_\X$, with $\pi_1\ll\pi_0$, and let $(\A,\phi)=(\PR^{\pi_0},\phi^{\pi_0})$. Then, we have
$$\sup_{U^\infty\text{ e-process}}\E_{\pi_1^\otimes}[R(\tau^{U^\infty})]=\sup_{A^\infty}\E_{\pi_1^\otimes}[R(\tau^{A^\infty})]=\sup_{a^\infty\text{ Markov}}\E_{\pi_1^\otimes}[R(\tau^{a^\infty})]\,.$$
\end{theorem*}
First, we state a lemma that ensures that the pair $(\A^{\pi_0},\phi^{\pi_0})$, defined in \Cref{sec:simple}, is a valid control framework for the simple-vs-simple case.
\begin{lemma}
\label{lemma:canonical_action_space}
Fix $\pi_0\in\PR_\X$ and define $\PR^{\pi_0}=\{\nu\in\PR_\X:\nu\ll\pi_0\}$. Then, $\PR^{\pi_0}$ is a standard Borel space. Moreover, there exists a Borel map $\phi^{\pi_0}:\PR^{\pi_0}\times\X\to\R_+$ such that, for every $\nu\in\PR^{\pi_0}$, the map $x\mapsto\phi^{\pi_0}(\nu,x)$ is a version of the Radon--Nikodym derivative $\dd\nu/\dd\pi_0$. Equivalently, for every Borel set $B\subseteq\X$,
$$\nu(B)=\int_B\phi^{\pi_0}(\nu,x)\,\dd\pi_0(x)\,.$$
\end{lemma}
\begin{proof}
We endow $\PR_\X$ with the $\sigma$-field generated by the evaluation maps $\nu\mapsto\nu(B)$, where $B\subseteq\X$ is Borel. Since $\X$ is standard Borel, $\PR_\X$ is standard Borel \citep[Theorem~17.23]{kechris1995classical}\footnote{The result in \cite{kechris1995classical} is stated for Polish spaces. We apply it to a Polish topology on $\X$ generating the given Borel $\sigma$-field. The resulting Borel $\sigma$-field on $\PR_\X$ is the one generated by the evaluation maps.}. Moreover, since $\X$ admits a Polish topology generating the same Borel $\sigma$-field, \cite[Theorem~3.5]{lange1973borel} implies that $\PR^{\pi_0}$ is a Borel subset of $\PR_\X$. Therefore $\PR^{\pi_0}$ is standard Borel.

For $\nu\in\PR^{\pi_0}$ and any Borel set $B\subseteq\X$, define $Q(\nu,B)=\nu(B)$. Then $Q$ is a probability kernel from $\PR^{\pi_0}$ to $\X$, since $Q(\nu,\cdot)=\nu$ and the evaluation maps $\nu\mapsto\nu(B)$ are Borel. Moreover, $Q(\nu,\cdot)\ll\pi_0$ for every $\nu\in\PR^{\pi_0}$. By \cite[Lemma~2.1]{kallenberg2014stationary}, there exists a Borel map $\rho:\PR^{\pi_0}\times\X\to[0,\infty)$ such that
$$Q(\nu,B)=\int_B\rho(\nu,x)\,\dd\pi_0(x)$$
for all $\nu\in\PR^{\pi_0}$ and every Borel set $B\subseteq\X$. Taking $\phi^{\pi_0}=\rho$ gives the claim.
\end{proof}

\begin{proof}[Proof of \Cref{thm:eprocess_equivalence}]
Let $(\A,\phi)=(\PR^{\pi_0},\phi^{\pi_0})$ as in \Cref{lemma:canonical_action_space}. We first show that every policy in the canonical testing-by-betting model induces an e-process. Fix a policy $A^\infty$. Under $\pi_0^\otimes$, the associated wealth process satisfies
$$\widehat W_t^{A^\infty}=\widehat W_{t-1}^{A^\infty}\phi^{\pi_0}(A_t(X^{t-1}),X_t)\,,$$
with $\widehat W_0^{A^\infty}=1$. Since $A_t(X^{t-1})$ is $\sigma(X^{t-1})$-measurable and
$$\int_\X\phi^{\pi_0}(\nu,x)\,\dd\pi_0(x)=1$$
for every $\nu\in\PR^{\pi_0}$, we have
$$\E_{\pi_0^\otimes}\left[\widehat W_t^{A^\infty}\mid X^{t-1}\right]=\widehat W_{t-1}^{A^\infty}\,.$$
Thus $\widehat W^{A^\infty}$ is a non-negative $\pi_0^\otimes$-martingale with initial value one, and hence an e-process. Therefore,
\begin{equation}\label{eq:ineqeasy}\sup_{A^\infty}\E_{\pi_1^\otimes}[R(\tau^{A^\infty})]\leq \sup_{U^\infty\text{ e-process}}\E_{\pi_1^\otimes}[R(\tau^{U^\infty})]\,.\end{equation}

We now prove the reverse inequality. Fix an arbitrary e-process $U^\infty=(U_t)_{t\geq0}$ under the point null $\pi_0^\otimes$. By \cite[Lemma~6]{ramdas2020admissible}, there exists a non-negative $\pi_0^\otimes$-martingale $M^\infty=(M_t)_{t\geq0}$, which we may take to satisfy $M_0=1$, such that $U_t\leq M_t$ for every $t\geq0$, $\pi_0^\otimes$-almost surely. Since $\pi_1\ll\pi_0$, we have $\pi_1^{\otimes t}\ll\pi_0^{\otimes t}$ for every finite $t$. Hence, intersecting over countably many times, $U_t\leq M_t$ for every $t\geq0$, $\pi_1^\otimes$-almost surely. It follows that $\tau^{M^\infty}\leq\tau^{U^\infty}$, $\pi_1^\otimes$-almost surely. Since $R$ is non-increasing,
$$\E_{\pi_1^\otimes}[R(\tau^{M^\infty})]\geq\E_{\pi_1^\otimes}[R(\tau^{U^\infty})]\,.$$

It remains to represent $M^\infty$ as a canonical wealth process. Since $M_t$ is $\sigma(X^t)$-measurable and $\X^t$ is standard Borel, the Doob--Dynkin lemma gives Borel maps $m_t:\X^t\to\R_+$ such that $M_t=m_t(X^t)$, $\pi_0^\otimes$-almost surely. We take $m_0\equiv1$. For $t\geq1$ and $x^{t-1}\in\X^{t-1}$, define
$$b_t(x^{t-1})=\int_\X m_t(x^{t-1},x)\,\dd\pi_0(x)\,.$$
By the martingale property, $b_t(X^{t-1})=m_{t-1}(X^{t-1})$, $\pi_0^{\otimes(t-1)}$-almost surely. Let
$$N_t=\{x^{t-1}\in\X^{t-1}:b_t(x^{t-1})\neq m_{t-1}(x^{t-1})\}\,.$$
Then $N_t$ is Borel and $\pi_0^{\otimes(t-1)}(N_t)=0$. Replacing $m_t$ by
$$\widetilde m_t(x^{t-1},x)=m_t(x^{t-1},x)\one_{N_t^c}(x^{t-1})+m_{t-1}(x^{t-1})\one_{N_t}(x^{t-1})\,,$$
we still have $M_t=\widetilde m_t(X^t)$, $\pi_0^\otimes$-almost surely, and now
$$\int_\X \widetilde m_t(x^{t-1},x)\,\dd\pi_0(x)=m_{t-1}(x^{t-1})$$
for every $x^{t-1}\in\X^{t-1}$.

Define
$$e_t(x^{t-1},x)=\begin{cases}\widetilde m_t(x^{t-1},x)/m_{t-1}(x^{t-1})&\text{if }m_{t-1}(x^{t-1})>0\,,\\ 1&\text{if }m_{t-1}(x^{t-1})=0\,.\end{cases}$$
Then $e_t$ is Borel, non-negative, and satisfies
$$\int_\X e_t(x^{t-1},x)\,\dd\pi_0(x)=1$$
for every $x^{t-1}\in\X^{t-1}$. Hence, for each $x^{t-1}\in\X^{t-1}$, the formula
$$\nu_t(x^{t-1},B)=\int_B e_t(x^{t-1},x)\,\dd\pi_0(x)$$
defines a probability measure $\nu_t(x^{t-1},\,\cdot\,)\in\PR^{\pi_0}$. Moreover, since $x^{t-1}\mapsto\nu_t(x^{t-1},B)$ is Borel for every Borel set $B\subseteq\X$, the map $x^{t-1}\mapsto\nu_t(x^{t-1},\,\cdot\,)$ is Borel from $\X^{t-1}$ to $\PR^{\pi_0}$.

Now define a policy $A^\infty$ by setting $A_t(x^{t-1})=\nu_t(x^{t-1},\,\cdot\,)$. By the defining property of $\phi^{\pi_0}$ in \Cref{lemma:canonical_action_space}, for each $x^{t-1}\in\X^{t-1}$ the map $x\mapsto\phi^{\pi_0}(A_t(x^{t-1}),x)$ is a version of $e_t(x^{t-1},\,\cdot\,)$, since $e_t(x^{t-1},\,\cdot\,)$ is a version of $\dd\nu_t(x^{t-1},\,\cdot\,)/\dd\pi_0$. Hence $\phi^{\pi_0}(A_t(x^{t-1}),x)=e_t(x^{t-1},x)$ for $\pi_0$-almost every $x$, and therefore also for $\pi_1$-almost every $x$, since $\pi_1\ll\pi_0$. By induction over $t$, we get
$$\widehat W_t^{A^\infty}=M_t$$
for every $t\geq0$, $\pi_1^\otimes$-almost surely. Consequently, $\tau^{A^\infty}=\tau^{M^\infty}$, $\pi_1^\otimes$-almost surely, and so
$$\E_{\pi_1^\otimes}[R(\tau^{A^\infty})]=\E_{\pi_1^\otimes}[R(\tau^{M^\infty})]\geq\E_{\pi_1^\otimes}[R(\tau^{U^\infty})]\,.$$
Taking the supremum over e-processes gives
$$\sup_{U^\infty\text{ e-process}}\E_{\pi_1^\otimes}[R(\tau^{U^\infty})]\leq\sup_{A^\infty}\E_{\pi_1^\otimes}[R(\tau^{A^\infty})]\,.$$

Combining this with \eqref{eq:ineqeasy} yields
$$\sup_{U^\infty\text{ e-process}}\E_{\pi_1^\otimes}[R(\tau^{U^\infty})]=\sup_{A^\infty}\E_{\pi_1^\otimes}[R(\tau^{A^\infty})]\,.$$
Finally, applying \Cref{thm:markovian_reduction} to the canonical action space $(\PR^{\pi_0},\phi^{\pi_0})$ gives
$$\sup_{A^\infty}\E_{\pi_1^\otimes}[R(\tau^{A^\infty})]=\sup_{a^\infty\text{ Markov}}\E_{\pi_1^\otimes}[R(\tau^{a^\infty})]\,,$$
and the claim follows.
\end{proof}

\subsection[Proof of Proposition 2]{Proof of \Cref{prop:hard-deadline-np}}
We provide here a proof of \Cref{prop:hard-deadline-np}, which we restate here.
\begin{proposition*}
Consider testing $\pi_0^\otimes$ against $\pi_1^\otimes$, with $\pi_1\ll\pi_0$. Then there exists a Borel set $\Gamma^\star\subseteq\X^T$, viewed as an event depending only on the first $T$ observations, such that $\pi_0^{\otimes}(\Gamma^\star)\leq\alpha$ and
$$\sup_{U^\infty\text{ e-process}}\pi_1^\otimes\bigl(\tau^{U^\infty}\leq T\bigr)=\pi_1^{\otimes}(\Gamma^\star)=\max\bigl\{\pi_1^{\otimes}(\Gamma)\,:\,\Gamma\subseteq\X^T\text{ Borel}\,,\,\pi_0^{\otimes}(\Gamma)\leq\alpha\bigr\}\,.$$
Fixing Borel versions of the conditional probabilities, the supremum  is achieved by the $\pi_0^\otimes$-martingale
$U_t^\star(x^t)=\pi_0^\otimes(\Gamma^\star\mid x^t)/\pi_0^\otimes(\Gamma^\star)$ if $\pi_0^\otimes(\Gamma^\star)>0$, and by the constant e-process $U_t\equiv 1$ otherwise.
\end{proposition*}

\begin{proof}
    First, we show that there exists $\Gamma^\star$ as required by the statement. Define the finite $\R^2$-valued measure $\mu$ on $\X^T$, given by
    $$\mu(\Gamma) = \big(\pi_0^{\otimes T}(\Gamma),\pi_1^{\otimes T}(\Gamma)\big)\,.$$
    Its range is defined as
    $$\mathcal R_\mu = \{\mu(\Gamma)\,:\,\Gamma\subseteq\X^T\;\text{Borel}\}\,.$$
    By Liapounoff's vector-measure theorem \citep[Theorem 2]{liapounoff1940fonctions}\footnote{Paper in Russian, with French summary at the end.}, $\mathcal R_\mu$ is closed in $\R^2$, and hence compact since it is bounded. In particular, the set
    $$\mathcal C = \mathcal R_\mu \cap \big((-\infty,\alpha]\times \R\big)$$ is compact and non-empty (as $(0,0)\in\mathcal C$). Hence, there exists $(p^\star, q^\star)\in\mathcal C$ such that
    $$q^\star = \max\{q\,:\,(p,q)\in\mathcal C\}\,.$$
    Then, there is a Borel $\Gamma^\star\in\X^T$ such that $\mu(\Gamma^\star) = (p^\star,q^\star)$. In particular, $\pi_0^{\otimes T}(\Gamma^\star) \leq \alpha$ and $$\pi_1^{\otimes T}(\Gamma^\star) = \max\{\pi_1^{\otimes T}(\Gamma)\,:\,\Gamma\subseteq\X^T\;\text{Borel}\,,\,\pi_0^{\otimes T}(\Gamma)\leq\alpha\}\,.$$

    We now show that this set gives the optimal hard-deadline value. Let $U^\infty$ be any e-process under $\pi_0^\otimes$, and define
    $$\Gamma_U=\left\{x^T\in\X^T:\max_{0\leq t\leq T}U_t(x^t)\geq\alpha^{-1}\right\}\,.$$
    Then $\Gamma_U$ is Borel and $\{\tau^{U^\infty}\leq T\}=\Gamma_U$. For the bounded stopping time $\tau^{U^\infty}\wedge T$, the e-process property gives $\E_{\pi_0^\otimes}\left[U_{\tau^{U^\infty}\wedge T}\right]\leq1$. On the event $\{\tau^{U^\infty}\leq T\}$ we have $U_{\tau^{U^\infty}\wedge T}\geq\alpha^{-1}$, and therefore
    $$\pi_0^{\otimes T}(\Gamma_U)=\pi_0^\otimes(\tau^{U^\infty}\leq T)\leq \pi_0^{\otimes T}(U_{\tau^{U^\infty}\wedge T}\geq\alpha^{-1})\leq\alpha\,,$$ by Markov's inequality.
    By the definition of $\Gamma^\star$, it follows that $\pi_1^\otimes(\tau^{U^\infty}\leq T)=\pi_1^{\otimes T}(\Gamma_U)\leq\pi_1^{\otimes T}(\Gamma^\star)$.
    Taking the supremum over all e-processes gives
    $$\sup_{U^\infty\text{ e-process}}\pi_1^\otimes(\tau^{U^\infty}\leq T)\leq\pi_1^{\otimes T}(\Gamma^\star)\,.$$

    It remains to show that this upper bound is attained. Suppose first that $\pi_0^{\otimes T}(\Gamma^\star)>0$. Fix Borel versions of the conditional probabilities
    $$x^t\mapsto\pi_0^\otimes(\Gamma^\star\mid X^t=x^t)$$
    for $0\leq t\leq T$, and define
    $$U_t^\star(x^t)=\frac{\pi_0^\otimes(\Gamma^\star\mid X^t=x^t)}{\pi_0^{\otimes T}(\Gamma^\star)}$$
    for $0\leq t\leq T$. For $t>T$, set $U_t^\star(x^t)=U_T^\star(x^T)$. Then $U^{\star,\infty}$ is a non-negative $\pi_0^\otimes$-martingale with initial value one, and hence an e-process. Moreover, for $x^T\in\Gamma^\star$,
    $$U_T^\star(x^T)=\frac{1}{\pi_0^{\otimes T}(\Gamma^\star)}\geq\alpha^{-1}\,,$$
    because $\pi_0^{\otimes T}(\Gamma^\star)\leq\alpha$. Hence $\Gamma^\star\subseteq\{x^T:\tau^{U^{\star,\infty}}(x^\infty)\leq T\}$, and so
    $$\pi_1^{\otimes T}(\Gamma^\star)\leq\pi_1^\otimes(\tau^{U^{\star,\infty}}\leq T)\,.$$
    Combining this with the upper bound already proved gives
    $$\pi_1^\otimes(\tau^{U^{\star,\infty}}\leq T)=\pi_1^{\otimes T}(\Gamma^\star)\,.$$

    Finally, if $\pi_0^{\otimes T}(\Gamma^\star)=0$, then $\pi_1^{\otimes T}(\Gamma^\star)=0$ since $\pi_1\ll\pi_0$. In this case the constant process $U_t^\star\equiv1$ is an e-process and satisfies
    $$\pi_1^\otimes(\tau^{U^{\star,\infty}}\leq T)=0=\pi_1^{\otimes T}(\Gamma^\star)\,.$$
    This proves the claim.
\end{proof}

\subsection[Proof of Proposition 3]{Proof of \Cref{prop:approximation-constant-strategy}}
Here we prove \Cref{prop:approximation-constant-strategy}, which we restate for convenience.
\begin{proposition*}
Let $(a_\star,\eta_\star)$ satisfy \eqref{eq:cond}, with $\eta_\star\in(0,1)$ and $\phi(a_\star,\,\cdot\,)\leq \Phi$, $\pi_1$-almost surely. Then
$$(\alpha/\Phi)^{\eta_\star}\leq \E_{\pi_1^\otimes}[e^{-\tau^{a_\star}/T}] \leq \sup_{a^\infty\text{ Markov}}\E_{\pi_1^\otimes}[e^{-\tau^{a^\infty}/T}]\leq \alpha^{\eta_\star}\,.$$
\end{proposition*}
We recall that \eqref{eq:cond}, mentioned in the statement, reads
$$\sup_{a\in\A}\E_{\pi_1}\bigl[\phi(a,\,\cdot\,)^{\eta_\star}\bigr] = \E_{\pi_1}\bigl[\phi(a_\star,\,\cdot\,)^{\eta_\star}\bigr] = e^{1/T}\,.$$
\begin{proof}
Fix a Markov policy $a^\infty$, let $\widetilde W_t^{a^\infty}$ be the wealth process starting from wealth one, and write
$$Z_t^{a^\infty}=\alpha \widetilde W_t^{a^\infty}\,,$$
so that rejection occurs when $Z_t^{a^\infty}\geq1$. Define
$$M_t=e^{-(t\wedge\tau^{a^\infty})/T}\big(Z_{t\wedge\tau^{a^\infty}}^{a^\infty}\big)^{\eta^\star}\,.$$
We first show that $(M_t)_{t\geq0}$ is a non-negative supermartingale under $\pi_1^\otimes$. On the event $\{\tau^{a^\infty}\leq t\}$, we have $M_{t+1}=M_t$. On the event $\{\tau^{a^\infty}>t\}$, the next action is $\sigma(X^t)$-measurable. Hence, using independence under $\pi_1^\otimes$ and the defining property of $(a^\star,\eta^\star)$,
$$\E_{\pi_1^\otimes}\left[M_{t+1}\mid X^t\right]\leq e^{-(t+1)/T}\big(Z_t^{a^\infty}\big)^{\eta^\star}\sup_{a\in\A}\E_{\pi_1}\left[\phi(a,X_1)^{\eta^\star}\right]=e^{-t/T}\big(Z_t^{a^\infty}\big)^{\eta^\star}=M_t\,.$$
Thus $(M_t)_{t\geq0}$ is a non-negative supermartingale. Therefore, for every $n\geq1$, $\E_{\pi_1^\otimes}[M_n]\leq M_0=\alpha^{\eta^\star}$. On the event $\{\tau^{a^\infty}\leq n\}$, we have $M_n=e^{-\tau^{a^\infty}/T}(Z_{\tau^{a^\infty}}^{a^\infty})^{\eta^\star}$ and $Z_{\tau^{a^\infty}}^{a^\infty}\geq1$, so
$$e^{-\tau^{a^\infty}/T}\one_{\{\tau^{a^\infty}\leq n\}}\leq M_n\,.$$
It follows that $\E_{\pi_1^\otimes}[e^{-\tau^{a^\infty}/T}\one_{\{\tau^{a^\infty}\leq n\}}]\leq\alpha^{\eta^\star}$. Sending $n\to\infty$ and using monotone convergence gives
$$\E_{\pi_1^\otimes}\left[e^{-\tau^{a^\infty}/T}\right]\leq\alpha^{\eta^\star}\,.$$
Taking the supremum over Markov policies yields
$$\sup_{a^\infty\text{ Markov}}\E_{\pi_1^\otimes}\left[e^{-\tau^{a^\infty}/T}\right]\leq\alpha^{\eta^\star}\,.$$

We now consider the constant policy that always plays $a^\star$. Let $\widetilde W_t^{a^\star}$ be the corresponding wealth process starting from wealth one, and write $Z_t^{a^\star}=\alpha\widetilde W_t^{a^\star}$. Define $N_t=e^{-t/T}(Z_t^{a^\star})^{\eta^\star}$. Then $(N_t)_{t\geq0}$ is a non-negative martingale under $\pi_1^\otimes$, because
$$\E_{\pi_1^\otimes}\left[N_{t+1}\mid X^t\right]=e^{-(t+1)/T}\big(Z_t^{a^\star}\big)^{\eta^\star}\E_{\pi_1}\left[\phi(a^\star,X_1)^{\eta^\star}\right]=N_t\,.$$
On the event $\{\tau^{a^\star}<\infty\}$, we have $Z_{\tau^{a^\star}-1}^{a^\star}<1$ and $Z_{\tau^{a^\star}}^{a^\star}\geq1$. Since $\phi(a^\star,\cdot)\leq\Phi$, $\pi_1$-almost surely, this implies $1\leq Z_{\tau^{a^\star}}^{a^\star}\leq\Phi$, $\pi_1^\otimes$-almost surely on $\{\tau^{a^\star}<\infty\}$. On the event $\{\tau^{a^\star}>n\}$, instead, $Z_n^{a^\star}<1$. Hence $0\leq Z_{\tau^{a^\star}\wedge n}^{a^\star}\leq\Phi$, $\pi_1^\otimes$-almost surely.

By optional stopping at the bounded stopping time $\tau^{a^\star}\wedge n$,
$$\alpha^{\eta^\star}=\E_{\pi_1^\otimes}\left[e^{-(\tau^{a^\star}\wedge n)/T}\big(Z_{\tau^{a^\star}\wedge n}^{a^\star}\big)^{\eta^\star}\right]\,.$$
Sending $n\to\infty$ and using dominated convergence yields
$$\alpha^{\eta^\star}=\E_{\pi_1^\otimes}\left[e^{-\tau^{a^\star}/T}\big(Z_{\tau^{a^\star}}^{a^\star}\big)^{\eta^\star}\one_{\{\tau^{a^\star}<\infty\}}\right]\,.$$
Since $Z_{\tau^{a^\star}}^{a^\star}\leq\Phi$ on $\{\tau^{a^\star}<\infty\}$, we get
$$\alpha^{\eta^\star}\leq\Phi^{\eta^\star}\E_{\pi_1^\otimes}\left[e^{-\tau^{a^\star}/T}\right]\,.$$
Therefore,
$$\left(\frac{\alpha}{\Phi}\right)^{\eta^\star}\leq\E_{\pi_1^\otimes}\left[e^{-\tau^{a^\star}/T}\right]\,.$$
Since the constant policy is Markov, the middle inequality is immediate. Combining the three bounds gives
$$(\alpha/\Phi)^{\eta^\star}\leq\E_{\pi_1^\otimes}\left[e^{-\tau^{a^\star}/T}\right]\leq\sup_{a^\infty\text{ Markov}}\E_{\pi_1^\otimes}\left[e^{-\tau^{a^\infty}/T}\right]\leq\alpha^{\eta^\star}\,,$$
as claimed.
\end{proof}

\subsection[Proof of Proposition 4]{Proof of \Cref{prop:exp-reward-summary}}
Here we prove \Cref{prop:exp-reward-summary}, restated below.
\begin{proposition*}
Let $G = \{(1-\eta)\log\E_{\pi_0}[L_\eta]\,:\,\eta\in[0,1)\,,\,\E_{\pi_0}[L_\eta]<\infty\}$. If $1/T\in(0,\sup G)$, then there exist $\eta_\star\in(0,1)$ and $\pi_\star\ll\pi_0$ satisfying \eqref{eq:cond}. Moreover,
$\dd\pi_\star/\dd\pi_0=L_{\eta_\star}/\E_{\pi_0}[L_{\eta_\star}]$
and $\log(\sup_{\pi\ll\pi_0}\E_{\pi_1}[(\dd\pi/\dd\pi_0)^{\eta_\star}]) = \eta^\star D_{1/(1-\eta_\star)}(\pi_1\|\pi_0) = 1/T$.
\end{proposition*}
We recall that \eqref{eq:cond}, mentioned in the statement, reads
$$\sup_{a\in\A}\E_{\pi_1}\bigl[\phi(a,\,\cdot\,)^{\eta_\star}\bigr] = \E_{\pi_1}\bigl[\phi(a_\star,\,\cdot\,)^{\eta_\star}\bigr] = e^{1/T}\,.$$
\begin{proof}
Let $L=L_0 =\dd\pi_1/\dd\pi_0$. We first show that there exists $\eta_\star\in(0,1)$ such that
$$(1-\eta_\star)\log\E_{\pi_0}[L_{\eta_\star}]=1/T\,.$$
Indeed, since $1/T\in(0,\sup G)$, there exists $\bar\eta\in(0,1)$ such that $\E_{\pi_0}[L_{\bar\eta}]<\infty$ and
$$(1-\bar\eta)\log\E_{\pi_0}[L_{\bar\eta}]>1/T\,.$$
The map $\eta\mapsto(1-\eta)\log\E_{\pi_0}[L_\eta]$ is finite and continuous on $[0,\bar\eta]$. Finiteness follows by interpolation, since $L_\eta=L^{1/(1-\eta)}$ and $1/(1-\eta)\leq1/(1-\bar\eta)$ for $\eta\leq\bar\eta$. Continuity follows from the continuity of $q\mapsto\|L\|_q$ on finite intervals where the largest moment is finite. Since the same map is equal to zero at $\eta=0$, the intermediate value theorem gives $\eta_\star\in(0,\bar\eta)$ such that
$$(1-\eta_\star)\log\E_{\pi_0}[L_{\eta_\star}]=1/T\,.$$

We now identify the maximiser in \eqref{eq:cond}. Fix $\eta\in(0,1)$ such that $\E_{\pi_0}[L_\eta]<\infty$, and write $p=1/\eta$ and $q=1/(1-\eta)$. Then $1/p + 1/q = 1$. Let $\pi\ll\pi_0$ and set $f=\dd\pi/\dd\pi_0$. By Hölder's inequality,
$$\E_{\pi_1}[f^\eta]=\E_{\pi_0}[Lf^\eta]\leq \E_{\pi_0}[L^q]^{1/q}\E_{\pi_0}[f]^\eta=\E_{\pi_0}[L_\eta]^{1-\eta}\,.$$
Thus
$$\sup_{\pi\ll\pi_0}\E_{\pi_1}\left[\left(\dd\pi/\dd\pi_0\right)^\eta\right]\leq \E_{\pi_0}[L_\eta]^{1-\eta}\,.$$
The upper bound is attained by the probability measure $\pi_\eta\ll\pi_0$ defined by
$$\frac{\dd\pi_\eta}{\dd\pi_0}=\frac{L_\eta}{\E_{\pi_0}[L_\eta]}\,.$$
Indeed, for this choice,
$$\E_{\pi_1}\left[\left(\frac{\dd\pi_\eta}{\dd\pi_0}\right)^\eta\right]=\E_{\pi_0}\left[L\left(\frac{L^q}{\E_{\pi_0}[L^q]}\right)^\eta\right]=\E_{\pi_0}[L^q]^{1-\eta}=\E_{\pi_0}[L_\eta]^{1-\eta}\,,$$
where we used $1+q\eta=q$. Applying this with $\eta=\eta_\star$ and setting $\pi_\star=\pi_{\eta_\star}$ gives
$$\frac{\dd\pi_\star}{\dd\pi_0}=\frac{L_{\eta_\star}}{\E_{\pi_0}[L_{\eta_\star}]}\,.$$
Moreover,
$$\log\left(\sup_{\pi\ll\pi_0}\E_{\pi_1}\left[\left(\dd\pi/\dd\pi_0\right)^{\eta_\star}\right]\right)=(1-\eta_\star)\log\E_{\pi_0}[L_{\eta_\star}]=1/T\,.$$
Finally, since
$$\eta_\star D_{1/(1-\eta_\star)}(\pi_1\|\pi_0)=(1-\eta_\star)\log\E_{\pi_0}[L_{\eta_\star}]\,,$$
we obtain
$$\log\left(\sup_{\pi\ll\pi_0}\E_{\pi_1}\left[\left(\dd\pi/\dd\pi_0\right)^{\eta_\star}\right]\right)=\eta_\star D_{1/(1-\eta_\star)}(\pi_1\|\pi_0)=1/T\,.$$
Thus $\pi_\star$ attains the supremum in \eqref{eq:cond}, and \eqref{eq:cond} holds.
\end{proof}

\subsection[Proof of Proposition 5]{Proof of \Cref{prop:power-constant-strategy}}
We prove \Cref{prop:power-constant-strategy}, which we now recall.
\begin{proposition*}
Fix $a\in\A$, assume that $\phi(a,\,\cdot\,)>0$, $\pi_1$-almost surely, and that $\gamma(a)=\E_{\pi_1}[\log\phi(a,\,\cdot\,)]\in[-\infty,\infty]$ is well-defined. Then, $\pi_1^\otimes(\tau_w^a<\infty)=1$ for every $w\in(0,\alpha^{-1})$ if and only if either $\gamma(a)>0$, or $\gamma(a)=0$ and $\pi_1(\log\phi(a,\,\cdot\,)>0)>0$. Moreover, when $\gamma(a)<0$, if $\phi(a,\,\cdot\,)\leq\Phi$, $\pi_1$-almost surely, and there exists $\kappa>0$ such that $\E_{\pi_1}[\phi(a,\,\cdot\,)^\kappa]=1$, then $(\alpha w/\Phi)^\kappa\leq\pi_1^\otimes(\tau_w^a<\infty)\leq(\alpha w)^\kappa$ for every $w\in(0,\alpha^{-1})$.
\end{proposition*}
\begin{proof}
Write $Y_i=\log\phi(a,X_i)$ and $S_t=\sum_{i=1}^tY_i$. Under $\pi_1^\otimes$, the variables $(Y_i)_{i\geq1}$ are i.i.d. with well-defined mean $\gamma(a)$. For initial wealth $w\in(0,\alpha^{-1})$, the constant policy that always plays $a$ has wealth
$$\widetilde W_t^{w,a}=w\exp(S_t)\,,$$
and therefore
$$\tau_w^a=\inf\{t\geq1:S_t\geq \log(\alpha^{-1}/w)\}\,.$$
Set $c_w=\log(\alpha^{-1}/w)>0$.

Suppose first that $\gamma(a)>0$. By the strong law of large numbers, $S_t\to+\infty$, $\pi_1^\otimes$-almost surely. Hence $S_t$ hits the level $c_w$ in finite time almost surely, and so $\pi_1^\otimes(\tau_w^a<\infty)=1$. If instead $\gamma(a)=0$ and $\pi_1(Y_1>0)>0$, then $(S_t)_{t\geq0}$ is a centred non-degenerate random walk. Since its increments have mean zero, it cannot drift to $+\infty$ or $-\infty$. By Feller's oscillation theorem for one-dimensional random walks \citep[Chapter~XII, Section~2, Theorem~1]{feller1971introduction}, $\limsup_t S_t=+\infty$ almost surely. Hence $S_t$ hits the level $c_w$ in finite time almost surely, and again $\pi_1^\otimes(\tau_w^a<\infty)=1$.

We now prove the converse. If $\gamma(a)=0$ and $\pi_1(Y_1>0)=0$, then $Y_1=0$, $\pi_1$-almost surely, so $S_t=0$ for every $t\geq0$ and the positive level $c_w$ is never reached. If $\gamma(a)<0$, then the strong law of large numbers gives $S_t\to-\infty$, $\pi_1^\otimes$-almost surely. In particular, $\sup_t S_t<\infty$ almost surely. We claim that $\pi_1^\otimes(\sup_t S_t<c_w)>0$. Indeed, choose $\varepsilon>0$ such that $\pi_1(Y_1\leq-\varepsilon)>0$. Since an independent copy of the random walk has finite supremum almost surely, we may choose $m$ large enough that this supremum is smaller than $c_w+m\varepsilon$ with positive probability. On the event that the first $m$ increments are all at most $-\varepsilon$ and the future supremum is smaller than $c_w+m\varepsilon$, the level $c_w$ is never hit. Thus $\pi_1^\otimes(\tau_w^a<\infty)<1$. This proves the first part of the proposition.

We now assume that $\gamma(a)<0$, that $\phi(a,\,\cdot\,)\leq\Phi$, $\pi_1$-almost surely, and that there exists $\kappa>0$ such that $\E_{\pi_1}[\phi(a,\,\cdot\,)^\kappa]=1$. Then $(\widetilde W_t^{w,a})^\kappa$ is a non-negative $\pi_1^\otimes$-martingale with initial value $w^\kappa$. Optional stopping at the bounded stopping time $\tau_w^a\wedge n$ gives
$$w^\kappa=\E_{\pi_1^\otimes}\left[(\widetilde W_{\tau_w^a\wedge n}^{w,a})^\kappa\right]\geq \alpha^{-\kappa}\pi_1^\otimes(\tau_w^a\leq n)\,.$$
Sending $n\to\infty$ yields
$$\pi_1^\otimes(\tau_w^a<\infty)\leq(\alpha w)^\kappa\,.$$

It remains to prove the lower bound. If $\Phi=\infty$, the bound is trivial, so assume $\Phi<\infty$. Fix $b\in(0,w)$ and define $\sigma_b=\inf\{t\geq0:\widetilde W_t^{w,a}\leq b\}$. Since $\gamma(a)<0$, we have $\widetilde W_t^{w,a}\to0$, $\pi_1^\otimes$-almost surely. Hence $\sigma_b<\infty$ almost surely for every $b\in(0,w)$, and $\sigma_b\to\infty$ almost surely as $b\downarrow0$. Optional stopping at $\tau_w^a\wedge\sigma_b\wedge n$ gives
$$w^\kappa=\E_{\pi_1^\otimes}\left[(\widetilde W_{\tau_w^a\wedge\sigma_b\wedge n}^{w,a})^\kappa\right]\,.$$
Letting $n\to\infty$ and splitting according to whether $\tau_w^a<\sigma_b$ or $\sigma_b<\tau_w^a$, we obtain
$$w^\kappa\leq \left(\frac{\Phi}{\alpha}\right)^\kappa\pi_1^\otimes(\tau_w^a<\sigma_b)+b^\kappa\,.$$
Here we used that, on $\{\tau_w^a<\infty\}$, the pre-rejection wealth is below $\alpha^{-1}$, and therefore the rejection wealth is at most $\Phi\alpha^{-1}$. Sending $b\downarrow0$ gives
$$w^\kappa\leq \left(\frac{\Phi}{\alpha}\right)^\kappa\pi_1^\otimes(\tau_w^a<\infty)\,,$$
and hence
$$(\alpha w/\Phi)^\kappa\leq\pi_1^\otimes(\tau_w^a<\infty)\,.$$
Combining the two bounds proves the claim.
\end{proof}

\newpage
\section{Details for the examples}\label{sec:app-examples}
\subsection{Gaussian example}\label{sec:app-gauss}
Let $\pi_0 = \N(0, \sigma^2)$ and $\pi_1 = \N(\mu, \sigma^2)$, with $\mu>0$ and $\sigma>0$. We consider sequential tests for $\pi_0^\otimes$ against $\pi_1^\otimes$. We let $\A$ and $\phi$ be as in \Cref{sec:simple}.

\paragraph{Hard deadline.} Fix an integer $T\geq 1$ and consider the reward $R=\one_{[0,T]}$. By the Neyman--Pearson lemma, the optimal fixed-horizon rejection event is obtained by spending the available null mass on the largest values of the likelihood ratio. Thus the event $\Gamma^\star$ from \Cref{prop:hard-deadline-np} can be chosen in the form $\Gamma^\star = \{\dd\pi_1^\otimes/\dd\pi_0^\otimes\geq C\}$, where $C$ is chosen so that $\pi_0^\otimes(\dd\pi_1^{\otimes T}/\dd\pi_0^{\otimes T}\geq C)=\alpha$.

We let $S_t(x^t)=\sum_{k=1}^t x_k$ and notice that
$$\frac{\dd\pi_1^{\otimes T}}{\dd\pi_0^{\otimes T}}(x^T)=\exp\left(\frac{\mu}{\sigma^2}S_T(x^T)-\frac{T\mu^2}{2\sigma^2}\right)\,.$$
Since $\mu>0$, this expression is strictly increasing in $S_T(x^T)$. In particular, we can write $\Gamma^\star=\{S_T\geq c\}$, where $c$ satisfies $\pi_0^\otimes(S_T\geq c)=\alpha$. Since $S_T\sim\N(0,T\sigma^2)$ under $\pi_0^\otimes$, we have
$$c=\sigma\sqrt T\,F^{-1}(1-\alpha)\,,$$
where $F$ is the cumulative distribution function of the standard normal.

The optimal e-process from \Cref{prop:hard-deadline-np} can be written as
$$U_t^\star(x^t)=\frac{\pi_0^\otimes(S_T\geq c\mid S_t)}{\alpha}\,.$$
Since under $\pi_0^\otimes$ we have $\Delta_{T,t}=S_T-S_t\sim\N(0,\sigma^2(T-t))$, independently of $S_t$, we obtain

$$U_t^\star(x^t)=\frac{h_t(S_t(x^t))}{\alpha}\,,$$
where, for $t<T$,
$$h_t(s)=1-F\left(\frac{c-s}{\sigma\sqrt{T-t}}\right)\,,$$
with terminal convention $h_T(s)=\one_{[c,\infty)}(s)$.

Notice that, for every $t<T$ and every $s\in\R$, the Gaussian increment $S_T-S_t$ has full support under $\pi_0^\otimes$. Hence
$$0<h_t(s)<1\,.$$
Consequently, for every $t<T$,
$$U_t^\star(x^t)=\frac{h_t(S_t(x^t))}{\alpha}<\alpha^{-1}\,.$$
Thus the exact hard-deadline Doob test does not reject before the deadline. It rejects at time $T$ precisely on the event $\Gamma^\star=\{S_T\geq c\}$, since
$$U_T^\star(x^T)=\alpha^{-1}\one_{\{S_T\geq c\}}\,.$$

Following the discussion after \Cref{prop:hard-deadline-np}, we now describe the canonical predictable \emph{truncated} strategy $A^T$ whose wealth process is $U^{\star,T}$ (for $t\geq T$, any continuation of the policy and wealth process are equivalent for the hard deadline objective). For $t=1,\ldots,T$ and $x^{t-1}\in\R^{t-1}$, the action $A_t(x^{t-1})$ is the measure
$$A_t(x^{t-1})(\dd x)=\frac{h_t(S_{t-1}(x^{t-1})+x)}{h_{t-1}(S_{t-1}(x^{t-1}))}\,\dd\pi_0(x)\,.$$

Remarkably, although this strategy is written as a function of the whole past, it corresponds to a Markovian strategy, since it can be made to depend on the current wealth only. Indeed, the wealth at round $t-1$ is precisely
$$U_{t-1}^\star(x^{t-1})=\frac{h_{t-1}(S_{t-1}(x^{t-1}))}{\alpha}\,.$$
In particular, we can write $W_{t-1}=w_{t-1}(S_{t-1}(x^{t-1}))$, where $w_{t-1}(s)=h_{t-1}(s)/\alpha$. Since $h_{t-1}$ is strictly increasing, we can invert this relation. Let $s_{t-1}=w_{t-1}^{-1}$. Then
$$S_{t-1}(x^{t-1})=s_{t-1}(W_{t-1})\,.$$
Since $A_t$ depends on $x^{t-1}$ only through $S_{t-1}$, the strategy depends only on current wealth. More precisely, the equivalent Markovian policy is $a^T$, where for $t=1,\ldots,T$ and wealth $w$ the action $a_t(w)$ is the measure
$$a_t(w)(\dd x)=\frac{h_t(s_{t-1}(w)+x)}     {h_{t-1}(s_{t-1}(w))}\,\dd\pi_0(x)=\frac{1}{\alpha w}h_t(s_{t-1}(w)+x)\,\dd\pi_0(x)\,,$$
where we used $h_{t-1}(s_{t-1}(w))=\alpha w$ by construction.

Finally, an important remark is that the value of $\mu$ does not appear in the construction, except through its sign, which here we have assumed to be positive. The magnitude of $\mu$ affects the power but not the rejection region or the associated Doob strategy. This contrasts with the GRO criterion, which would play the constant action $\pi_1$.

\paragraph{Exponentially decaying reward.} Fix $T>0$ and let $R(t)=e^{-t/T}$. We now compute the EDO action from \Cref{prop:exp-reward-summary}. Plugging the closed form
$$D_\xi(\pi_1\|\pi_0)=\frac{\xi\mu^2}{2\sigma^2}$$
of the Rényi divergence into the characteristic relation defining $\eta_\star$, we obtain
$$\frac{1}{1-\eta_\star}=1+\frac{2\sigma^2}{\mu^2T}\,.$$
Hence
$$L_{\eta_\star}(x)=\left(\frac{\dd\pi_1}{\dd\pi_0}(x)\right)^{1/(1-\eta_\star)}=\exp\left(\left(\frac{\mu}{\sigma^2}+\frac{2}{\mu T}\right)x-\left(\frac{\mu^2}{2\sigma^2}+\frac{1}{T}\right)\right)\,.$$
By \Cref{prop:exp-reward-summary}, the optimal constant action $\pi_\star$ has Lebesgue density proportional to $(\dd\pi_0/\dd x)L_{\eta_\star}$. Therefore
$$\frac{\dd\pi_\star}{\dd x}(x)\propto \exp\left(-\frac{x^2}{2\sigma^2}+\left(\frac{\mu}{\sigma^2}+\frac{2}{\mu T}\right)x\right)\propto \exp\left(-\frac{1}{2\sigma^2}\left(x-\mu-\frac{2\sigma^2}{\mu T}\right)^2\right)\,.$$
Thus
$$\pi_\star=\N\left(\mu+\frac{2\sigma^2}{\mu T},\,\sigma^2\right)\,.$$
Compared with GRO, which plays $\pi_1=\N(\mu,\sigma^2)$, EDO is more aggressive, shifting further away from the null. As $T\to\infty$, this additional shift vanishes and $\pi_\star$ recovers the GRO action $\pi_1$. Finally, we remark that by \Cref{prop:power-constant-strategy}, this increased aggressiveness preserves power one if and only if
$$T\geq\frac{1}{\KL{\pi_1}{\pi_0}}=\frac{2\sigma^2}{\mu^2}\,.$$

\subsection{Bernoulli example}\label{sec:app-bern}
Let $\pi_0 = \Bernoulli(p_0)$ and $\pi_1=\Bernoulli(p_1)$, with $0<p_0<p_1<1$. We consider sequential tests for $\pi_0^\otimes$ against $\pi_1^\otimes$. We let $\A$ and $\phi$ be as in \Cref{sec:simple}.
\paragraph{Hard deadline.} As in the Gaussian example we define $S_t(x^t) = \sum_{k=1}^t x_k$. Again, the likelihood ratio $\dd\pi_1^{\otimes T}/\dd\pi_0^{\otimes T}$ depends on $x^T$ only through $S_T$, since we have
$$\frac{\dd\pi_1^{\otimes T}}{\dd\pi_0^{\otimes T}}(x^T) = \left(\frac{p_1}{p_0}\right)^{S_T(x^T)}\left(\frac{1-p_1}{1-p_0}\right)^{T-S_T(x^T)}\,.$$
Since $p_1>p_0$, this expression is increasing in $S_T(x^T)$. However, unlike in the Gaussian case, $S_T$ is discrete. Hence it might be that no level $\{S_T=c\}$ has exactly probability $\alpha$ under $\pi_0^\otimes$. For this reason, in general the Neyman--Pearson event $\Gamma^\star$ (see \Cref{prop:hard-deadline-np}) is not $S_T$-measurable.

To get a concrete example of the issue, take $T=3$, $p_0=1/2$, $p_1=3/4$, and $\alpha=1/4$. Under $\pi_0^{\otimes 3}$, every triple of observations has mass $1/8$, while the likelihood ratio is increasing in $S_3$. The best $S_3$-measurable upper-tail event is $\{S_3=3\}=\{111\}$\footnote{Here, $111$ is the triple of observations $x_1=1$, $x_2=1$, $x_3=1$.}, since $\pi_0^\otimes(S_3\geq 2)=4/8>\alpha$, and its power is
$$\pi_1^\otimes(S_3=3)=\left(\frac34\right)^3=\frac{27}{64}\,.$$
However, the unrestricted Neyman--Pearson event can include one tied boundary sequence, for example
$$\Gamma^\star=\{111,110\}\,.$$
Then $\pi_0^\otimes(\Gamma^\star)=2/8=\alpha$, while
$$\pi_1^\otimes(\Gamma^\star)=\frac{27}{64}+\frac{9}{64}=\frac{36}{64}\,.$$
Thus the exact deterministic Neyman--Pearson event need not be $S_T$-measurable: restricting to $S_T$-measurable upper tails can lose power in the discrete case.

In order to identify $\Gamma^\star$, it is enough to decide how many sequences to select from each level set of $S_T$. For each $k=0,\ldots,T$, let $m_k$ denote the number of sequences selected from $\{S_T=k\}$. Since both $\pi_0^\otimes$ and $\pi_1^\otimes$ assign the same mass to all sequences in a fixed level, the Neyman--Pearson problem reduces to the finite integer programme
$$\max_{m_0,\ldots,m_T}\sum_{k=0}^T m_k p_1^k(1-p_1)^{T-k}$$
subject to
$$\sum_{k=0}^T m_k p_0^k(1-p_0)^{T-k}\leq \alpha\,,\qquad 0\leq m_k\leq {T\choose k}\,,\quad m_k\in\Nn\,.$$
After fixing a deterministic tie-break within each level, for instance lexicographic order with $1>0$, any maximiser $(m_0^\star,\dots,m_T^\star)$ defines an exact deterministic Neyman--Pearson event $\Gamma^\star$ by selecting the first $m_k^\star$ sequences from each level $\{S_T=k\}$.

When $p_0\leq 1/2$, the integer programme has a simpler greedy structure. Indeed, the null mass of a single sequence with $k$ ones is $p_0^k(1-p_0)^{T-k}$, which is non-increasing in $k$. Hence, after all levels with larger likelihood ratio have been included, if the remaining level budget is too small to add another sequence from the current boundary level, then it is also too small to add any sequence from a lower level. Thus one may include all sequences above a boundary level and then fill the remaining budget with as many sequences as possible from that boundary level. More precisely, choose $k_\alpha$ such that
$$\pi_0^\otimes(S_T>k_\alpha)\leq \alpha<\pi_0^\otimes(S_T\geq k_\alpha)\,,$$
and set
$$r_\alpha=\left\lfloor\frac{\alpha-\pi_0^\otimes(S_T>k_\alpha)}{p_0^{k_\alpha}(1-p_0)^{T-k_\alpha}}\right\rfloor\,.$$
After fixing the same deterministic tie-break within $\{S_T=k_\alpha\}$, an exact Neyman--Pearson event is
$$\Gamma^\star=\{S_T>k_\alpha\}\cup \mathcal C_{k_\alpha,r_\alpha}\,,$$
where $\mathcal C_{k_\alpha,r_\alpha}$ denotes the first $r_\alpha$ sequences in the level $\{S_T=k_\alpha\}$.

Once an optimal $\Gamma^\star$ is determined, the optimal e-process is again the Doob martingale from \Cref{prop:hard-deadline-np}. More explicitly, set $h_0=\pi_0^\otimes(\Gamma^\star)$ and, for $t=0,\ldots,T$,
$$h_t(x^t)=\pi_0^\otimes(\Gamma^\star\mid x^t)=\sum_{y\in\{0,1\}^{T-t}}\one_{\Gamma^\star}(x^t,y)\,p_0^{S_{T-t}(y)}(1-p_0)^{T-t-S_{T-t}(y)}\,.$$
Then
$$U_t^\star(x^t)=\frac{h_t(x^t)}{h_0}\,.$$
Following the discussion after \Cref{prop:hard-deadline-np}, this e-process is implemented by the predictable strategy which, whenever $h_{t-1}(x^{t-1})>0$, plays the action
$$A_t(x^{t-1})(\{x\})=\frac{h_t(x^{t-1},x)}{h_{t-1}(x^{t-1})}\,\pi_0(\{x\})\,,\qquad x\in\{0,1\}\,.$$
When $h_{t-1}(x^{t-1})=0$, the action can be chosen arbitrarily. The recursion $h_{t-1}(x^{t-1})=\sum_{x\in\{0,1\}}h_t(x^{t-1},x)\pi_0(\{x\})$ shows that $A_t(x^{t-1})$ is a probability measure, and the corresponding wealth telescopes to $U_t^\star$. Contrarily to the Gaussian case, in general this strategy is not Markovian in the wealth.

For large $T$, however, explicitly identifying $\Gamma^\star$ and computing the conditional quantities $h_t(x^t)$ quickly becomes computationally costly, since the exact construction may require keeping track of a tie-broken subset of a boundary level. In such cases, a practical simplification is to use the $S_T$-measurable upper-tail event obtained by ignoring the partial boundary correction. Namely, let
$$\widetilde k=\min\{k:\pi_0^\otimes(S_T\geq k)\leq\alpha\}\,,$$
and set $\widetilde\Gamma=\{S_T\geq \widetilde k\}$. This event is optimal among $S_T$-measurable upper-tail events, and it coincides with the exact Neyman--Pearson event whenever the level constraint aligns with a binomial upper tail. Otherwise, the discrepancy is confined to the boundary level. Heuristically, this correction becomes less important for large $T$: when $p_0$ is bounded away from $0$ and $1$, the mass of any single boundary atom is exponentially small in $T$, so ignoring a partial boundary selection has only a small effect on the rejection probability.

For this simplified event, the Doob quantities depend only on the current number of ones. More precisely, for $s\in\{0,\ldots,t\}$,
$$\widetilde h_t(s)=\pi_0^\otimes(S_T\geq \widetilde k\mid S_t=s)=\sum_{j=(\widetilde k-s)_+}^{T-t}{T-t\choose j}p_0^j(1-p_0)^{T-t-j}\,.$$
The corresponding predictable action is therefore
$$\widetilde A_t(x^{t-1})(\{x\})=\frac{\widetilde h_t(S_{t-1}(x^{t-1})+x)}{\widetilde h_{t-1}(S_{t-1}(x^{t-1}))}\,\pi_0(\{x\})\,,\qquad x\in\{0,1\}\,.$$
This gives a simple $S_T$-based implementation of the hard-deadline Doob strategy, which is the approximation we use when the exact boundary selection is computationally impractical.

\paragraph{Exponentially decaying reward.} Fix $T>0$ and let $R(t)=e^{-t/T}$. We now compute the EDO action from \Cref{prop:exp-reward-summary}. For Bernoulli distributions, the Rényi divergence has the closed form
$$D_\xi(\pi_1\|\pi_0)=\frac{1}{\xi-1}\log\left(p_1^\xi p_0^{1-\xi}+(1-p_1)^\xi(1-p_0)^{1-\xi}\right)\,.$$
Plugging this into the characteristic relation defining $\eta_\star$, namely $\eta_\star D_{1/(1-\eta_\star)}(\pi_1\|\pi_0)=1/T$, gives the scalar equation
$$ (1-\eta_\star)\log\left(p_1^{1/(1-\eta_\star)}p_0^{-\eta_\star/(1-\eta_\star)}+(1-p_1)^{1/(1-\eta_\star)}(1-p_0)^{-\eta_\star/(1-\eta_\star)}\right)=\frac{1}{T}\,.$$
In contrast to the Gaussian case, this equation does not simplify to a closed-form expression for $\eta_\star$, but it is only one-dimensional and can easily be solved numerically.

We now set
$$L_{\eta_\star}(x)=\left(\frac{\dd\pi_1}{\dd\pi_0}(x)\right)^{1/(1-\eta_\star)}\,.$$
Thus
$$L_{\eta_\star}(1)=\left(\frac{p_1}{p_0}\right)^{1/(1-\eta_\star)}\,,\qquad L_{\eta_\star}(0)=\left(\frac{1-p_1}{1-p_0}\right)^{1/(1-\eta_\star)}\,.$$
By \Cref{prop:exp-reward-summary}, the EDO action $\pi_\star$ is the $L_{\eta_\star}$-tilt of $\pi_0$. Hence $\pi_\star=\Bernoulli(p_\star)$, where
\begin{equation}\label{eq:pstar}p_\star=\frac{p_0\left(\frac{p_1}{p_0}\right)^{1/(1-\eta_\star)}}{p_0\left(\frac{p_1}{p_0}\right)^{1/(1-\eta_\star)}+(1-p_0)\left(\frac{1-p_1}{1-p_0}\right)^{1/(1-\eta_\star)}}\,.\end{equation}
Equivalently,
$$\log\frac{p_\star}{1-p_\star}=\frac{1}{1-\eta_\star}\log\frac{p_1}{1-p_1}-\frac{\eta_\star}{1-\eta_\star}\log\frac{p_0}{1-p_0}\,.$$
Since $p_1>p_0$ and $\eta_\star>0$, this tilt moves past the alternative, so $p_\star>p_1$. Compared with  GRO, which plays $\pi_1=\Bernoulli(p_1)$, EDO is therefore more aggressive. As $T\to\infty$, we have $\eta_\star\to0$, and this additional tilt vanishes, so $\pi_\star$ recovers the GRO action $\pi_1$.

Finally, by \Cref{prop:power-constant-strategy}, this increased aggressiveness preserves power one if and only if
$$T\geq\frac{1}{\KL{\pi_1}{\pi_0}}=\left(p_1\log\frac{p_1}{p_0}+(1-p_1)\log\frac{1-p_1}{1-p_0}\right)^{-1}\,.$$

\newpage
\section{Experimental details}
\label{sec:appendix-experiments}

\subsection[Details for Figure 1]{Details for \Cref{fig:heatmaps-Gaussian}}\label{sec:app-fig-heatmaps-gaussian}
\begin{center}
\includegraphics[width=\textwidth]{Plots/optimal_bet_heatmap_gaussian_lambda.png}
\end{center}

\Cref{fig:heatmaps-Gaussian} is obtained by solving the finite-horizon Bellman recursion numerically for a restricted Gaussian-shift action class. We take $\pi_0=\N(0,1)$, $\pi_1=\N(0.6,1)$ and $\alpha=0.05$, and use actions $a\in[0,4]$, corresponding to one-step e-variables $E_{[a]}(x)=\exp(ax-a^2/2)$. Thus the figure shows the Bellman optimum within this parametric class, not the unrestricted optimum over $\PR^{\pi_0}$. In particular, for the hard-deadline reward, the exact unrestricted Gaussian Doob strategy discussed in \Cref{sec:examples} would not reject before the deadline.

The Bellman recursion is solved over the rescaled wealth $z=\alpha w\in[0,1]$. We use a logarithmic grid for $z$ on $[10^{-8},1]$ with $401$ points, and a grid of $401$ actions. Expectations under $\pi_1$ are evaluated by Gauss--Hermite quadrature with $41$ nodes. To start the recursion, we have to truncate the rewards at a large value; in \cref{fig:heatmaps-Gaussian} the truncation is at $150$. This barely changes the objectives since, for all rewards, the values above $t=150$ are below $10^{-6}$.
The displayed panels show the early part of the policy, where the dependence on time and wealth is most visible.

The reward panels correspond to the hard-deadline reward $R(t)=\one_{t\leq 30}$, the logistic reward $R(t)=(1+\exp((t-30)/2))^{-1}$, and the exponential reward $R(t)=e^{-t/10}$. The heatmaps show the Bellman-optimal Gaussian-shift action $a$ as a function of time $t$ and log-wealth $\log w$. The ``hopeless region'' marks states where the reward has value zero, and hence any valid rejection time has passed so that further betting is hopeless.

The three panels illustrate how the reward shape changes the optimal betting behaviour. Under the hard-deadline reward, the policy becomes increasingly aggressive as the deadline approaches, reflecting the fact that rejection after time $30$ has no value. The logistic reward smooths this transition, leading to a more gradual change in actions around the same time scale. The exponential reward has no sharp deadline, and therefore produces a more stationary-looking policy away from the rejection boundary. The lower row shows the corresponding stopping-time CDFs under the alternative, estimated from $5000$ paths using the same quadrature dynamics as in the Bellman computation.

\newpage
\subsection[Details for Figure 2]{Details for \Cref{fig:bernoulli_doob}}\label{sec:appendix-fig-bernoulli_doob}

\begin{center}
\includegraphics[width=\textwidth]{Plots/bernoulli_hard_deadline_vs_gro_stacked_small.png}
\end{center}

\Cref{fig:bernoulli_doob} compares the hard-deadline Doob e-processes from \Cref{prop:hard-deadline-np} with GRO in the Bernoulli case. We take $\pi_0=\Bernoulli(0.4)$, $\pi_1=\Bernoulli(0.6)$ and $\alpha=0.05$, and consider deadlines $T=10$ and $T=20$. For each deadline, the terminal Neyman--Pearson event is constructed exactly. Writing $S_T=\sum_{t=1}^T x_t$, we choose the smallest $k$ such that $\pi_0^{\otimes T}(S_T\ge k)\leq\alpha$. Since the Bernoulli model is atomic, the deterministic level-$\alpha$ event also includes a subset of the tied boundary level $S_T=k-1$. We order boundary strings lexicographically with $1>0$, and include the first
$$r=\left\lfloor\frac{\alpha-\pi_0^{\otimes T}(S_T\ge k)}{p_0^{k-1}(1-p_0)^{T-k+1}}\right\rfloor\,.$$
This gives $k=8$ and $r=106$ for $T=10$, and $k=13$ and $r=102809$ for $T=20$.

For the resulting event $\Gamma^\star$, the Doob e-process is $U_t^\star=\pi_0^{\otimes T}(\Gamma^\star\mid x^t)/\pi_0^{\otimes T}(\Gamma^\star)$. Its stopping-time distribution under $\pi_1^{\otimes}$ is computed exactly, by propagating the $\pi_1$-mass of alive prefixes and removing a prefix as soon as it already forces membership in $\Gamma^\star$. The GRO curve is computed in the same way, using the likelihood-ratio wealth, whose log at a state with $s$ successes after $t$ rounds is $s\log(p_1/p_0)+(t-s)\log((1-p_1)/(1-p_0))$. The top panel aggregates the probability mass of $\tau$ into time bins, while the bottom panel shows the corresponding CDF under the alternative.

\newpage
\subsection[[Details for Figure 3]{Details for \Cref{fig:edo-approx}}\label{sec:appendix-fig-edo-approx}
\begin{center}
\includegraphics[width=\textwidth]{Plots/bern_approx.png}
\end{center}

\Cref{fig:edo-approx} compares EDO with a numerical Bellman solution in the Bernoulli case. We take $\pi_0=\Bernoulli(1/2)$, $\pi_1=\Bernoulli(2/3)$ and $\alpha=0.05$, and use the exponentially decaying reward $R(t)=e^{-t/T}$ with time scale $T=30$. The EDO action is the constant action $\pi_\star=\Bernoulli(p_\star)$ obtained from \Cref{prop:exp-reward-summary}. In this case $\eta_\star$ is computed numerically by solving $\eta_\star D_{1/(1-\eta_\star)}(\pi_1\|\pi_0)=1/T$, and $p_\star$ is then obtained from the closed-form Bernoulli tilt \eqref{eq:pstar}.

The Bellman curve is computed by finite-horizon backward induction on the rescaled wealth $z=\alpha w$. We use horizon $180$, plot rejection times up to $120$, and discretise the state and action spaces using $401$ wealth grid points and $401$ actions. The wealth grid is logarithmic on $[10^{-10},1]$. The Bellman policy uses action-level overshoot capping, meaning that if an outcome would already cross the boundary, its payoff is capped at the value needed to reach $z=1$. We compare it with the uncapped EDO action and with the same EDO action after applying this boundary-aware capping.

The stopping-time distributions are estimated under $\pi_1^{\otimes}$ using $300000$ Monte Carlo paths, with common random numbers across the three strategies. The top panel aggregates the probability mass of $\tau$ into time bins, while the bottom panel shows the corresponding CDF under the alternative.

\newpage
\subsection[Details for Figure 4]{Details for \Cref{fig:edo-gro}}\label{sec:appendix-fig-edo-gro}
\begin{center}
\includegraphics[width=\textwidth]{Plots/GROvsEDO.png}
\end{center}

\Cref{fig:edo-gro} compares GRO with the EDO strategy in the Gaussian case (no capping). We take $\pi_0=\N(0,1)$, $\pi_1=\N(1/4,1)$ and $\alpha=0.05$. GRO corresponds to the constant Gaussian-shift action $a=\mu=1/4$. For each time scale $T\in\{8,14,60\}$, EDO uses the constant action from \Cref{sec:app-gauss}, namely $a_\star(T)=\mu+2/(\mu T)$. Thus smaller time scales lead to larger shifts away from the null, favouring earlier rejection. As $T$ increases, $a_\star$ approaches the GRO action $\mu$, which is asymptotically optimal for minimising the first-order rejection-time scale as $\alpha\to0$. This is reflected in the CDFs: smaller time scales rise earlier but plateau below one, while the larger time scale more closely tracks GRO, which has the best large-time behaviour.

The stopping-time distributions are estimated under $\pi_1^{\otimes}$ using $300000$ Monte Carlo paths, with common random numbers across all strategies. The horizon is $250$, and rejection times are plotted up to $120$. No capping is applied. The top panel aggregates the probability mass of $\tau$ into time bins, while the bottom panel shows the corresponding CDF under the alternative.

The power effect is visible through \Cref{prop:power-constant-strategy}. For a constant Gaussian-shift action $a$, the log-growth under $\pi_1$ is $\gamma(a)=a\mu-a^2/2$. Hence power one holds exactly when $0<a\leq 2\mu$. For EDO this condition is equivalent to $T\geq 2/\mu^2$. In the present experiment, $2/\mu^2=32$, so the EDO curve with $T=60$ has power one, whereas the smaller time scales $T=8$ and $T=14$ trade eventual power for earlier rejection.

\newpage
\subsection[Details for Figure 5]{Details for \Cref{fig:power-bound-check}}\label{sec:appendix-fig-power-bound-check}
\begin{center}
\includegraphics[width=\textwidth]{Plots/power_proposition_bernoulli.png}
\end{center}

\Cref{fig:power-bound-check} illustrates the power statement in \Cref{prop:power-constant-strategy} for the Bernoulli EDO constant strategy. We take $\pi_0=\Bernoulli(0.5)$, $\pi_1=\Bernoulli(0.7)$ and $\alpha=0.05$. For each time scale $T$, we compute the EDO action by solving $\eta_\star D_{1/(1-\eta_\star)}(\pi_1\|\pi_0)=1/T$. This determines the constant e-variable $E_\star=d\pi_\star/d\pi_0$, which in the Bernoulli case is specified by its two values $E_\star(0)$ and $E_\star(1)$.

The drift under the alternative is $\gamma=(1-p_1)\log E_\star(0)+p_1\log E_\star(1)$. By \Cref{prop:power-constant-strategy}, the strategy has power one when $\gamma\geq0$. In this example the transition occurs at $T=\KL{\Bernoulli(0.7)}{\Bernoulli(0.5)}^{-1}$, where
$$\KL{\Bernoulli(0.7)}{\Bernoulli(0.5)}=0.7\log(1.4)+0.3\log(0.6)\,.$$
This threshold is marked by the vertical dotted line.

For $T$ below the threshold, the drift is negative. We then compute $\kappa>0$ solving $\E_{\pi_1}[E_\star(X)^\kappa]=1$, and plot the bounds from \Cref{prop:power-constant-strategy}, namely $(\alpha/B_\star)^\kappa$ and $\alpha^\kappa$, where $B_\star=\max\{E_\star(0),E_\star(1)\}$. The power itself is estimated under $\pi_1^{\otimes}$ by Monte Carlo, using $5000$ paths for each value of $T$ below the threshold. The error bars are Wilson confidence intervals for this Monte Carlo estimate, with the upper endpoint enlarged by the truncation bound used for paths stopped far below the rejection boundary. For $T$ above the threshold, power is set to one by the proposition.

\newpage
\section{Further discussion on the hard-deadline reward}\label{sec:appendix-hard}

We expand here on the hard-deadline construction of \Cref{prop:hard-deadline-np}. In the simple-vs-simple case, the reward $R(t)=\one_{\crl{t\leq T}}$ \Cref{prop:hard-deadline-np} reduces the problem to a fixed-horizon Neyman--Pearson problem at time $T$. Namely, one first chooses a terminal event $\Gamma^\star\subseteq\X^T$ satisfying $\pi_0^\otimes(\Gamma^\star)\leq\alpha$ and maximising $\pi_1^\otimes(\Gamma)$ among all terminal events $\Gamma\subseteq\X^T$ with $\pi_0^\otimes(\Gamma)\leq\alpha$. The corresponding sequential test is then obtained by inducing this terminal event through the Doob martingale $$U_t^\star(x^t)=\pi_0^\otimes(\Gamma^\star\mid x^t)/\pi_0^\otimes(\Gamma^\star)\,.$$ Thus the exact stopping rule is $$\pi_0^\otimes(\Gamma^\star\mid x^t)\geq \pi_0^\otimes(\Gamma^\star)/\alpha\,.$$ This is the sense in which the hard-deadline optimum has a terminal character: the event being optimised lives at time $T$, and the sequential test rejects earlier only if the observations so far make this terminal event sufficiently likely under the null.

When $\pi_0^{\otimes T}$ is non-atomic, one can choose the Neyman--Pearson set $\Gamma^\star$ so that $\pi_0^{\otimes T}(\Gamma^\star)=\alpha$: since there are no atoms, the likelihood-ratio rejection region can be adjusted to have exactly the desired null probability. In this case $U_T^\star=\alpha^{-1}\one_{\Gamma^\star}$, so the induced Doob test rejects at the deadline exactly on the fixed-sample Neyman--Pearson event. Before the deadline, the stopping rule becomes $\pi_0^\otimes(\Gamma^\star\mid x^t)\geq 1$, and hence early rejection is possible only if the observations $x^t$ already imply that the final path will belong to $\Gamma^\star$, whatever the remaining observations may be.

This phenomenon becomes very explicity for Gaussian mean testing, where rejection before the deadline cannot happen, as discussed in \Cref{sec:examples}. Let $\pi_0=\N(0,\sigma^2)$ and $\pi_1=\N(\mu,\sigma^2)$, with $\mu>0$. Then the fixed-horizon Neyman--Pearson region is $\Gamma^\star=\crl{S_T\geq c}$, where $S_T=\sum_{t=1}^T x_t$ and $c$ is chosen so that $\pi_0^\otimes(S_T\geq c)=\alpha$. For every $t<T$, conditional on $x^t$, the remaining sum $S_T-S_t$ has full support on $\R$ under $\pi_0^\otimes$. Therefore $\pi_0^\otimes(S_T\geq c\mid x^t)<1$ for every finite history $x^t$. Thus the exact Gaussian hard-deadline test never rejects before the deadline: it rejects at $T$ precisely on $\Gamma^\star$, and hence coincides with the fixed-sample Neyman--Pearson test.

In Bernoulli testing, by contrast, the fixed-horizon problem is atomic. Since the likelihood ratio is increasing in $S_T$, the Neyman--Pearson ordering favours paths with more successes. If randomisation were allowed, this would give the usual upper-tail rule, with randomisation at a boundary value of $S_T$. For deterministic events, however, the atoms cannot be split arbitrarily, and the exact set $\Gamma^\star$ is obtained by a finite integer optimisation over paths, or equivalently over the numbers of strings selected at each value of $S_T$. Except for special values of $\alpha$ for which the desired level is attained exactly by a union of atoms, such a deterministic choice may satisfy $\pi_0^\otimes(\Gamma^\star)<\alpha$. This creates a minor distinction between scaling the Doob martingale by $\pi_0^\otimes(\Gamma^\star)$ and scaling it by $\alpha$. In either case, since future observations are bounded, sufficiently favourable observations before time $T$ can force membership in the terminal event, and hence allow earlier rejection. Nevertheless, the hard-deadline objective still assigns the same reward to all rejection times before $T$, and so does not itself favour rejection substantially earlier than the deadline. Consequently, even in Bernoulli examples the induced Doob test can remain conservative until late in the horizon. This is the phenomenon referred to in \Cref{sec:hard}, and is one motivation for considering smoother rewards, such as the exponential-decay reward of \Cref{sec:exponential-reward}.

\Cref{fig:hard_deadline_bernoulli_comparison} compares the exact hard-deadline Doob construction with several betting rules in the Bernoulli example $\pi_0=\Bernoulli(0.4)$, $\pi_1=\Bernoulli(0.6)$ and $\alpha=0.05$. The Doob curve is computed exactly from the deterministic Neyman--Pearson event with lexicographic tie-breaking, using the stopping rule $U_t^\star\geq\alpha^{-1}$, equivalently $\pi_0^\otimes(\Gamma^\star\mid x^t)\geq \pi_0^\otimes(\Gamma^\star)/\alpha$. As expected from its terminal-power objective, this curve is the most conservative among the horizon-aware hard-deadline procedures: rejection mass is pushed towards the deadline. At the same time, it attains the largest probability of rejection by the deadline, as guaranteed by \Cref{prop:hard-deadline-np}. STaR-Bets from \cite{voravcek2025star} performs particularly well in this example, achieving substantially earlier rejection while remaining close to the Doob benchmark near the deadline, and typically outperforming the method from \cite{taga2026learning}.

The comparison between GRO and EDO isolates the effect of changing the constant betting rule. GRO plays the likelihood-ratio bet, equivalently $\Bernoulli(p_1)$, at every round. EDO instead plays the uncapped Bernoulli tilt $\Bernoulli(p^\star)$ obtained from the exponential-decay criterion with time scale equal to the plotted horizon. For the shorter horizons, EDO outperforms GRO, betting more aggressively and achieving earlier rejection. For $T=100$, the two curves are nearly indistinguishable, in line with the fact that EDO recovers GRO as the time scale grows.

For EDO, $p^\star$ is obtained as described in \Cref{sec:app-examples}. STaR-Bets is simulated under $\pi_1^\otimes$ using the finite-horizon target-recalculating rule of \citet{voravcek2025star}, with the empirical second-moment proxy regularised by the horizon-dependent bonus and clipped at $p_0(1-p_0)$ as in the paper's implementation. The Taga et al.\ curve is a simulated horizon-aware betting heuristic in the spirit of \citet{taga2026learning}: we use the empirical Kelly stake as a baseline, mix it with the endpoint stake according to six linear or quadratic schedules with onsets $0.25T,0.5T,0.75T$, and reject using the average wealth across these schedules. Both simulated curves use $20000$ Monte Carlo paths from $\pi^\otimes_1$ and the same random seed across horizons.

\begin{figure}[!t]
    \centering
    \includegraphics[width=\textwidth]{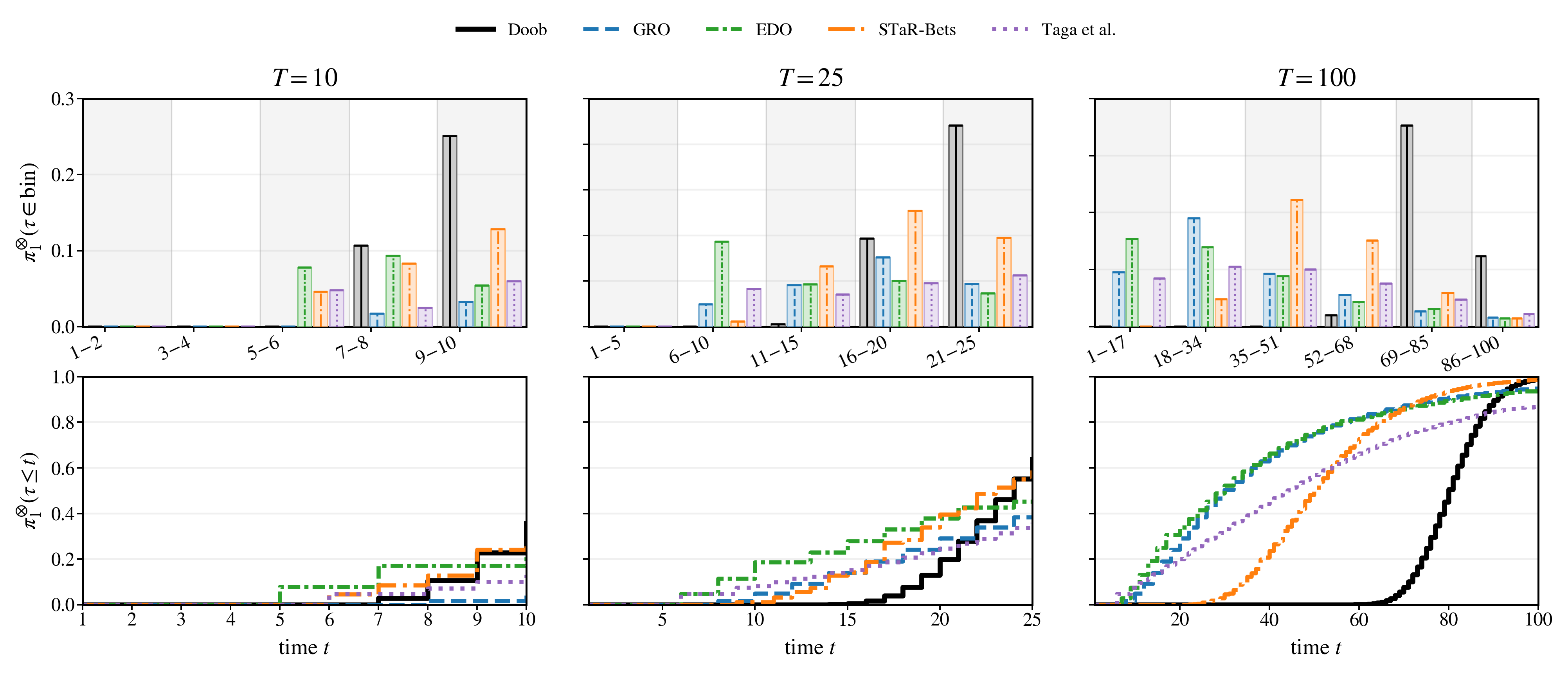}
    \caption{Rejection-time distributions for $\pi_0=\Bernoulli(0.4)$ and $\pi_1=\Bernoulli(0.6)$, with $\alpha=0.05$, for deadlines $T=10,25,100$. The top row shows coarse-binned rejection probabilities and the bottom row shows the corresponding cdfs. The Doob curve is the exact hard-deadline Doob/Neyman--Pearson benchmark, GRO and EDO are  constant-bet strategies, and STaR-Bets and Taga et al.\ are simulated horizon-aware baselines.}
    \label{fig:hard_deadline_bernoulli_comparison}
\end{figure}

\end{document}